\documentclass{article}

\usepackage[
	a4paper,
	margin=25mm
]{geometry}
\usepackage{setspace}
\usepackage{sectsty} 
\usepackage{titlesec}

\usepackage{amsmath, amssymb, amsthm}
\usepackage{mathtools}

\usepackage[
	setpagesize=false,
	colorlinks=true,
	linkcolor=blue,
        citecolor=blue,
        urlcolor=black,
	pdfencoding=auto,
	psdextra,
]{hyperref}
\usepackage{cleveref}

\usepackage{etoolbox} 
\usepackage[shortlabels]{enumitem}
\usepackage{xparse} 

\usepackage{graphicx}
\usepackage[dvipsnames]{xcolor}
\usepackage{tikz}

\usepackage{todonotes}
\usepackage[
    deletedmarkup=sout,
    commentmarkup=uwave,
    authormarkuptext=name
]{changes}

\usepackage{comment}

\setlist[enumerate,1]{label=(\arabic*), ref=(\arabic*)}
\setlist[enumerate,3]{label=(\roman*), ref=(\roman*)}

\allsectionsfont{\boldmath} 

\titlelabel{\thetitle.\quad}

\usepackage{caption}
\usepackage{subcaption}

\theoremstyle{plain}
\newtheorem{theorem}{Theorem}[section]
\newtheorem{lemma}[theorem]{Lemma}
\newtheorem{corollary}[theorem]{Corollary}
\newtheorem{proposition}[theorem]{Proposition}
\newtheorem{observation}[theorem]{Observation}
\newtheorem{conjecture}[theorem]{Conjecture}

\newtheorem{question}[theorem]{Question}

\newtheorem{claim}{Claim}[theorem]
\newtheorem*{claim*}{Claim}

\makeatletter
\newenvironment{claimproof}[1][Proof]{\par
	\pushQED{\qed}%
	
	\normalfont \topsep6\p@\@plus6\p@\relax
	\trivlist
	\item[\hskip\labelsep
	\textit{#1}\@addpunct{.}~]\ignorespaces
}{%
	\popQED\endtrivlist\@endpefalse
}
\makeatother

\newlist{Cases}{enumerate}{3}
\setlist[Cases]{parsep=0pt plus 1pt}
\setlist[Cases,1]{wide=0pt, listparindent=\parindent,
    label = \textbf{Case~\arabic*:}, ref = \arabic*}
\setlist[Cases,2]{wide=\parindent, listparindent=\parindent,
    label = \textbf{Case~\arabic{Casesi}-\arabic{Casesii}:}}

\crefname{Casesi}{case}{cases}

\newcounter{case}
\AtBeginEnvironment{proof}{\setcounter{case}{0}}

\crefname{case}{case}{cases}

\theoremstyle{definition}
\newtheorem{definition}[theorem]{Definition}

\newcommand{\calC}{\mathcal{C}}

\newcommand{\calF}{\mathcal{F}}

\newcommand{\calH}{\mathcal{H}}

\newcommand{\calP}{\mathcal{P}}

\newcommand{\calT}{\mathcal{T}}

\newcommand{\ve}{\varepsilon}

\makeatletter
\NewDocumentCommand{\xsideset}{mmme{_^}}{%
  \mathop{%
    \settowidth{\dimen0}{$\m@th\displaystyle#3$}%
    \dimen0=.5\dimen0
    \settowidth{\dimen2}{$%
      \m@th\displaystyle#3%
      \IfValueT{#4}{_{#4}}%
      \IfValueT{#5}{^{#5}}%
    $}%
    \dimen2=.5\dimen2
    \advance\dimen2 -\dimen0
    \sbox6{\scriptspace\z@$\displaystyle{\vphantom{#3}}#1$}
    \sbox8{\scriptspace\z@$\displaystyle{\vphantom{#3}}#2$}
    \ifdim\wd6>\dimen2 \kern\dimexpr\wd6-\dimen2\relax\fi
    {%
     \mathop{\llap{\copy6}{\displaystyle#3}\rlap{\copy8}}\limits
     \IfValueT{#4}{_{#4}}%
     \IfValueT{#5}{^{#5}}%
    }%
    \ifdim\wd8>\dimen2 \kern\dimexpr\wd8-\dimen2\relax\fi
  }%
}
\makeatother

\newcommand{\defeq}{\coloneqq}

\let\originalleft\left
\let\originalright\right
\renewcommand{\left}{\mathopen{}\mathclose\bgroup\originalleft}
\renewcommand{\right}{\aftergroup\egroup\originalright}

\makeatletter
\newcases{lrcases*}
  {\quad}
  {$\m@th{##}$\hfil}
  {{##}\hfil}
  {\lbrace}
  {\rbrace}
\makeatother

\usetikzlibrary{matrix,arrows,decorations.pathmorphing}
\usetikzlibrary{positioning}
\usetikzlibrary{graphs} 
\usetikzlibrary{graphs.standard}

\title{Towards a high-dimensional Dirac's theorem}

\author{
Hyunwoo Lee%
        \thanks{Department of Mathematical Sciences, KAIST, South Korea and Extremal Combinatorics and Probability Group (ECOPRO), Institute for Basic Science (IBS).
        E-mail: {\ttfamily hyunwoo.lee@kaist.ac.kr.} Supported by the National Research Foundation of Korea (NRF) grant funded by the Korea government(MSIT) No. RS-2023-00210430, and the Institute for Basic Science (IBS-R029-C4).}
}

\usepackage[
    backend=biber,
    sorting=nyt,
    maxbibnames=99
]{biblatex}
\begin{document}
\maketitle

\begin{abstract}
    Dirac's theorem determines the sharp minimum degree threshold for graphs to contain perfect matchings and Hamiltonian cycles. There have been various attempts to generalize this theorem to hypergraphs with larger uniformity by considering hypergraph matchings and Hamiltonian cycles. 
    In this paper, we consider another natural generalization of perfect matchings, Steiner triple systems. As a Steiner triple system can be viewed as a partition of pairs of vertices, it is a natural high-dimensional analogue of a perfect matching in graphs.
    We prove that for sufficiently large integer $n$ with $n \equiv 1 \text{ or } 3 \pmod{6}$, any $n$-vertex $3$-uniform hypergraph $H$ with minimum codegree at least $\left(\frac{3 + \sqrt{57}}{12} + o(1) \right)n = (0.879... + o(1))n$ contains a Steiner triple system. In fact, we prove a stronger statement by considering transversal Steiner triple systems in a collection of hypergraphs. 
    We conjecture that the number $\frac{3 + \sqrt{57}}{12}$ can be replaced with $\frac{3}{4}$ which would provide an asymptotically tight high-dimensional generalization of Dirac's theorem.
    
\end{abstract}


\section{Introduction}\label{sec:intro}

High-dimensional combinatorics initiated by Linial and other researchers have recently drawn much attention (see~\cite{linial2018challenges}) considering high-dimensional generalizations of many classical results. For example, a classical theorem of Dirac~\cite{dirac1952some} asserts that an $n$-vertex graph $G$ with the minimum degree $\delta(G) \geq \frac{n}{2}$ contains a perfect matching provided that $n$ is even. This theorem has been generalized to hypergraphs in several ways regarding the existence of hypergraph perfect matching, for instance, in~\cite{alon2012large,han2009perfect,lu2023co,rodl2009perfect,treglown2012exact,you2023minimum}.
While perfect matchings in hypergraphs provide a partition of $0$-dimensional objects (vertices), it is also very natural to consider a partition of higher-dimensional objects. In fact, Linial~\cite{linial2018challenges} mentioned that a Steiner triple system can be considered as a $2$-dimensional generalization of a graph perfect matching. This poses the following very natural question.

\begin{question}\label{ques:mincodegree}
    What is the minimum codegree condition on a $3$-uniform hypergraph $H$ that guarantees $H$ contains a Steiner triple system as a subgraph?     
\end{question}


\subsection{Combinatorial designs}\label{subsec:designs}

What is a Steiner triple system? A design with parameters $(n, q, r, \lambda)$ is a collection of $q$-element subsets of $n$ vertices such that every set of $r$ vertices is contained in exactly $\lambda$ subsets. We say that $(n, q, r, 1)$-designs are Steiner systems, and especially, $(n, 3, 2, 1)$-designs are called Steiner triple system. We will abbreviate `Steiner triple system' as $STS$. 
For a design with given parameters to exist, the divisibility condition on $n$ should be equipped. It is easy to see that an $(n, q, r, \lambda)$-design exists only if $\binom{q-i}{r-i}$ divides $\lambda \binom{n-i}{r-i}$ for every $0\leq i \leq r-1$. This natural condition is called the `divisibility condition'. The divisibility condition for $STS$ is $n \equiv 1 \text{ or }3 \pmod{6}$. Indeed, this is also a sufficient condition for the existence of the $STS$. Kirkman (1846) and independently Steiner (1853) proved that for every $n$ satisfying the divisibility condition, there is a $STS$ on $n$ vertices. This raises the natural question of whether $(n, q, r, \lambda)$-designs exist for all types $(n, q, r, \lambda)$ satisfying the divisibility condition and when $n$ is sufficiently large. This question is called the existence conjecture and it had been unsolved for over a hundred years until Keevash's resolution~\cite{keevash2014existence}.

The existence of Steiner systems when $r = 2$ was solved by Wilson~\cite{wilson1972existence1, wilson1972existence2, wilson1975existence3}. Wilson's theorem can be viewed as an edge decomposition of a complete graph. For a graph $F$, every sufficiently large $n$ with suitable divisibility condition, the edges of a complete graph $K_n$ can be decomposed into isomorphic copies of $F$. This decomposition is called an $F$-decomposition. More generally, the existence of designs is equivalent to the existence of hypergraph edge decomposition problems. An $(n, q, r, \lambda)$-design provides an edge decomposition of $n$-vertex $r$-uniform complete multi-hypergraph with all edge multiplicity $\lambda$ into copies of $r$-uniform complete hypergraph $K^{(r)}_q$ on $q$ vertices.

Keevash~\cite{keevash2014existence} resolved the existence conjecture by proving that for given parameters $q, r, \lambda$ and sufficiently large $n$ equipped with the divisibility condition, every hypergraph $K_n^{(r)}$ with multiplicity $\lambda$ admits an edge decomposition into $K_q^{(r)}$. These were generalized to $F$-decompositions for any $r$-uniform hypergraph $F$ by Glock, K\"{u}hn, Lo and Osthus~\cite{glock2023existence} via iterative absorption method.

Note that both results in~\cite{glock2023existence,keevash2014existence} can be generalized to show that near-complete hypergraphs (hypergraphs with almost maximum possible minimum codegree) also admit such decompositions as long as an appropriate ``$F$-divisibility condition" is met.
In~\cite{keevash2018existence}, Keevash further proved stronger decomposition theorems, proving that for every sufficiently large $n$ equipped with a suitable divisibility condition, $K_n^q$ can be decomposed into $(n, q, r, \lambda)$-designs.
Again this theorem can be generalized to near-regular hypergraphs if some additional codegree-regularity condition is given. Then it is also very natural to ask about the minimum codegree threshold for the existence of $(n, q, r, \lambda)$-design in $n$-vertex $q$-uniform hypergraph, in particular, \Cref{ques:mincodegree}. Note that a hypergraph matching in $3$-uniform hypergraph is a $(n, 3, 1, 1)$-design and $STS$ is an $(n, 3, 2, 1)$-design.

We prove the following theorem which provides a good upper bound on the codegree threshold that we are interested in.

\begin{theorem}\label{thm:simple-main}
    For any positive real number $\ve > 0$, there is $n_0 = n_0(\ve)$ that satisfies the following for all $n \geq n_0$ satisfying $n\equiv 1\text{ or } 3 \pmod{6}$.
    If $H$ is an $n$-vertex $3$-uniform hypergraph where the minimum codegree of $H$ is at least $\left(\frac{3 + \sqrt{57}}{12} + \ve \right)n$, then $H$ contains a $STS$.
\end{theorem}


\subsection{Transversals}\label{subsec:transversal}

Indeed, our main result (\Cref{thm:main}) is stronger than \Cref{thm:simple-main} as it provides a transversal copy of $STS$ for a given collection of hypergraphs. In order to properly state our main theorem, we introduce the concept of transversals over a hypergraph collection.

Let $\calF$ be the collection of sets $\{F_1, \dots, F_m\}$. We say a set $X = \{x_1, \dots, x_m\}$ is $\calF$-transversal if $F_i$ contains $x_i$ for each $i\in [m]$. The existence of transversals in various mathematical objects has been extensively studied in the various literature of mathematics. For example, in discrete geometry, many variations of Carath\'{e}odory's theorem~\cite{barany1982generalization,kalai2009colorful} and colorful Helly's theorem~\cite{kalai2005topological} were considered. For Matroid theory, Rota's basis conjecture~\cite{huang1994relations,pokrovskiy2020rota} concerns the transversals in the matroid basis. In design theory, many problems were studied related to Ryser's conjecture on transversals in Latin squares~\cite{keevash2022new,gould2023hamilton}. 

Recently, in (hyper)graph theory, finding a sufficient condition of the existence of transversals over a collection of (hyper)graphs was actively investigated (for instance, see ~\cite{joos2020rainbow,chakraborti2023bandwidth,cheng2023transversals,keevash2007rainbow,chakraborti2023hamilton,you2023minimum,aharoni2018rainbow}). 
Indeed, the notion of transversals was implicitly used in numerous literature even before its explicit introduction. The explicit notion of graph transversals was first defined in~\cite{joos2020rainbow} and extended to directed graphs and hypergraphs in~\cite{chakraborti2023hamilton} as follows.

\begin{definition}
    Let $\calF = \{F_1, \dots F_m\}$ be a collection of graphs/hypergraphs/digraphs on the common vertex set $V$. For a graph/hypergraph/digraph $G$, we say $G$ is $\calF$-transversal if there is a bijective function $\phi : E(G) \to [m]$ such that for every $e\in E(G)$, we have $e\in E(F_{\phi(e)})$. 
\end{definition}

As we are interested in $STS$ in this article, we only consider a collection of $3$-uniform hypergraphs throughout the paper.


\subsection{Fractional relaxations}\label{subsec:fractional-relaxation}

As we mentioned in \Cref{subsec:designs}, the existence of $STS$ is equivalent to an edge-decomposition of $K_n$ into triangles. For a graph $G$, we say $G$ is \emph{$K_3$-divisible} if all the vertices of $G$ have even degrees and $|E(G)|$ is divisible by $3$. 
The famous conjecture of Nash-Williams~\cite{nash1970unsolved} states that for sufficiently large $n$, if a given $K_3$-divisible graph $G$ on $n$ vertices has minimum degree at least $\frac{3}{4}n$, then $G$ can be edge-decomposed into triangles. While the Nash-Williams conjecture is widely open, the breakthrough result of Barber, K\"{u}hn, Lo, and Osthus~\cite{barber2016edge} essentially reduced the conjecture into a problem on fractional decomposition.

For the collection $\calT$ of triangles in a graph $G$, a \emph{fractional triangle-decomposition} of $G$ is a function $w: \calT \to [0, 1]$ such that $\sum_{e\in T\in \calT} W(T) = 1$ for every $e\in E(G)$. Let $\delta^*$ to be the infimum of $\delta \in [0, 1]$ that satisfies the following: there is $n_0$ such that for every graph $G$ on $n \geq n_0$ vertices with minimum degree at least $\delta n$ has a fractional triangle-decomposition. Barber, K\"{u}hn, Lo, and Osthus~\cite{barber2016edge} proved the following theorem stating that determination of $\delta^*$ would resolve an asymptotic form of Nash-Williams conjecture.

\begin{theorem}[Barber et al.~\cite{barber2016edge}]\label{thm:nash-williams}
    For any positive real number $\ve > 0$, there is $n_0 = n_0(\ve)$ that satisfies the following for all $n\geq n_0$.
    Let $G$ be a $K_3$-divisible $n$-vertex graph with minimun degree at least $\left(\max\{\frac{3}{4}, \delta^*\} + \ve \right)n$. Then $G$ has a triangle decomposition.
\end{theorem}

Since we consider $3$-uniform hypergraph decomposition, we define a fractional $STS$ on a given graph $G$ as follows.
Let $H$ be a hypergraph. We define \emph{$2$-shadow} $H^{(2)}$ of $H$ to be a graph with $V(H^{(2)}) = V(H)$ and $E(H^{(2)}) = \{uv\in \binom{V(H)}{2} : \{u, v\}\subseteq f \text{ for some } f\in E(H)\}$.
A \emph{fractional $STS$ of $H$} is a function $\phi : E(H)\to [0, 1]$ such that $\sum_{e\subset f\in E(H)} \phi(f) \leq 1$ for every $e\in E(H^{(2)})$. We say this is \emph{perfect} if $\sum_{e\subset f\in E(H)} \phi(f) = 1$ holds for all $e\in H^{(2)}$. In other words, a fractional $STS$ of $H$ is a fractional triangle decomposition of $H^{(2)}$ whose support is a subset of $E(H)$. We define a threshold function for perfect fractional $STS$ as below.

\begin{definition}
    A function $\delta_f: [0, 1)\to [0, 1]$ is defined as follows. For a real number $\ve \in [0, 1)$, the value $\delta_f(\ve)$ is the infimum of the set of numbers $\delta \in [0, 1]$ with the following property: there is $n_0$ such that for every $3$-uniform hypergraph $H$ on $n \geq n_0$ vertices with $\delta(H^{(2)}) \geq (1 - \ve)(n - 1)$ and codegree of $e$ in $H$ is at least $\delta (n - 2)$ for all $e\in E(H^{(2)})$ contains a perfect fractional $STS$ of $H$. 
    We denote by $\delta_f^*$ the right limit of $\delta_f(\ve)$ as $\ve \to 0$.
\end{definition}

Since a complete $3$-uniform hypergraph obviously contains a fractional $STS$, the function $\delta_f(\cdot)$ is well-defined. By the definition, the function $\delta_f(\ve)$ is bounded and non-increasing as $\ve \to 0$. Thus, by the monotone convergence theorem, $\delta_f^*$ exists. Note that we don't know whether $\delta_f(0)$ and $\delta^*_f$ have the same value.


\subsection{Our results}\label{subsec:results}

Our main result is the following theorem proving the existence of transversal $STS$ in a collection of sufficiently dense $3$-uniform hypergraphs.

\begin{theorem}\label{thm:main}
    For any positive real number $\ve > 0$, there is $n_0 = n_0(\ve)$ that satisfies the following for all $n \geq n_0$ satisfying $n\equiv 1\text{ or } 3 \pmod{6}$.
    Let $\calH = \{H_1, \dots, H_{\frac{n(n-1)}{6}}\}$ be a collection of $3$-uniform hypergraphs on the common $n$-vertex set $V$. Assume for each $i\in [\frac{n(n-1)}{6}]$, the minimum codegree of $H_i$ is at least $\left(\max\{\frac{3}{4}, \delta^*_f\} + \ve\right)n$. Then there is a $\calH$-transversal $STS$ on $V$.
\end{theorem}

We remark that \Cref{thm:main} is the first attempt to obtain the minimum codegree threshold not only for the existence of $(n, q, 2, 1)$-designs but also for the existence of transversal $(n, q, 2, 1)$-designs.

Indeed, implicitly in~\cite{glock2023existence}, it was proven that there exists $c > 0$ such that for sufficiently large $n$ satisfying the divisibility condition, every $n$-vertex $3$-uniform hypergraph $H$ with minimum codegree at least $(1 - c)n$ contains a $STS$. However, the value $c$ was not explicitly determined and not practical. \Cref{thm:main} determines a reasonable minimum codegree threshold not only for the existence of $STS$ but also for transversal $STS$. We observe that the value $\frac{3}{4}$ in the coefficient $\max\{\frac{3}{4}, \delta^*_f\}$ is the best possible.

\begin{proposition}\label{prop:lower-bound}
    For every sufficiently large $n$ with $n \equiv 1 \text{ or } 3 \pmod{6}$, there is an $n$-vertex $3$-uniform hypergraph $H$ with minimum codegree at least $\frac{3}{4}n - 26$ but has no $STS$.
\end{proposition}

\begin{proof}
    We construct such a hypergraph $H$. Let $V$ be a set of $n$ points and let $V = V_0 \cup V_1 \cup V_2 \cup V_3$ be a vertex partition with $\frac{n}{4} - 8 \leq |V_i|$ for each $0 \leq i \leq 3$ and $|V_0|$ is even and $|V_i|$ are odd for all $1 \leq i \leq 3$. 
    Let $E_0 = \{e\in \binom{V}{3}: |e\cap V_0| = 2\}$, $E_1 = \{e\in \binom{V}{3}: \text{ $e$ intersects each $V_0$, $V_i$, $V_j$ for some $1 \leq i \neq j \leq 3$}\}$. Let $E_2 = \{e\in \binom{V}{3}: e\subset V_i \text{ for some $1 \leq i \leq 3$}\}$ and $E_3 = \{e\in \binom{V}{3}: |e \cap V_i| = 2,\text{ } |e \cap V_j| = 1 \text{ for some $1 \leq i \neq j \leq 3$}\}$.
    We construct $H$ as $V(H) = V$ and $E(H) = E_0 \cup E_1 \cup E_2 \cup E_3$. Then the minimum codegree of $H$ is at least $\frac{3}{4}n - 26$. Thus, it suffices to show $H$ has no $STS$.

    \begin{figure}[ht]
     \centering
     \begin{subfigure}[b]{0.2\columnwidth}
         \centering
         \includegraphics[width=\columnwidth]{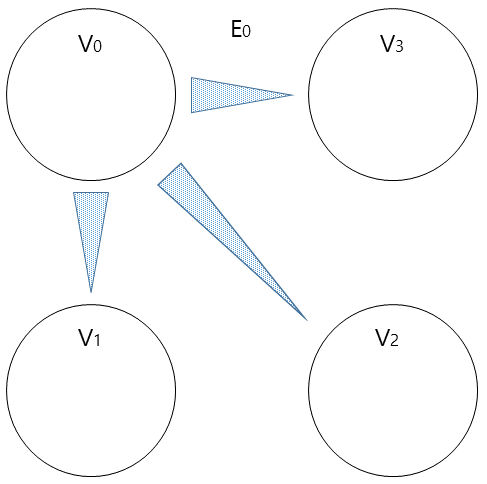}
         \caption{$E_0$}
     \end{subfigure}
     \hfill
     \begin{subfigure}[b]{0.2\columnwidth}
         \centering
         \includegraphics[width=\columnwidth]{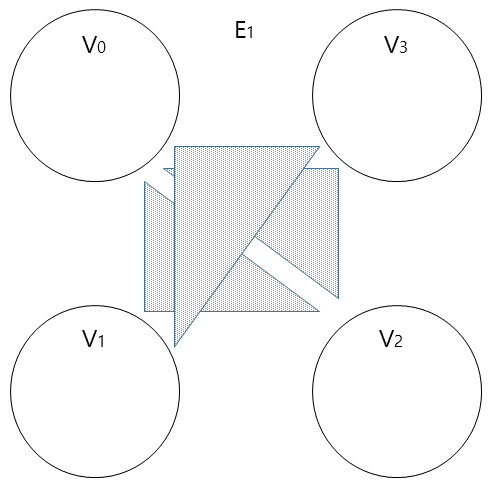}
         \caption{$E_1$}
     \end{subfigure}
     \hfill
     \begin{subfigure}[b]{0.2\columnwidth}
         \centering
         \includegraphics[width=\columnwidth]{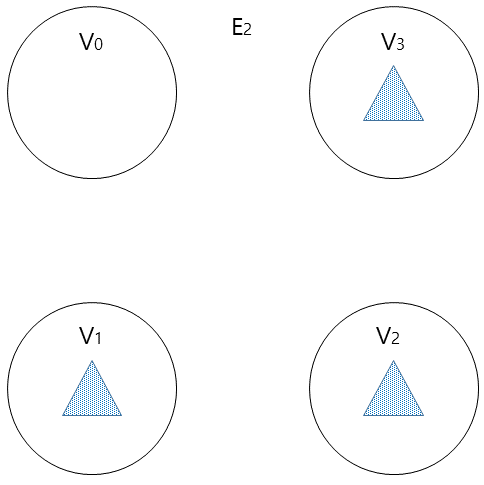}
         \caption{$E_2$}
     \end{subfigure}
    \hfill
     \begin{subfigure}[b]{0.2\columnwidth}
         \centering
         \includegraphics[width=\columnwidth]{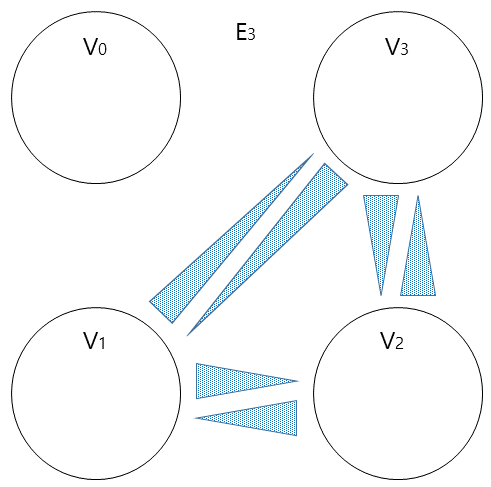}
         \caption{$E_3$}
     \end{subfigure}
     \caption{Vertex subsets and edge subsets of $H$.}
    \end{figure}

    For the sake of contradiction, assume $H$ contains a $STS$. Let $S$ be the such $STS$ in $H$. We denote $E'_i$ as $E(S)\cap E_i$ for each $0 \leq i \leq 3$. Since there is no edge in $V_0$, every pair of distinct vertices in $V_0$ must covered by an edge in $E'_0$. Since each edge in $E'_0$ covers two distinct pairs of vertices between $V_0$ and $V \setminus V_0$, the edges in $E'_0$ cover precisely $V_0(V_0 - 1)$ pairs of vertices between $V_0$ and $V\setminus V_0$. Then the number of uncovered pairs between $V_0$ and $V\setminus V_0$ is $|V_0| (n - 2|V_0| + 1)$. Then these edges should covered by $E'_1$. Since each edge in $E'_1$ covers two pairs of vertices between $V_0$ and $V\setminus V_0$ while covering one pair of vertices between $V_i$ and $V_j$ for some $1 \leq i \neq j \leq 3$. Thus, the number of pair of vertices between $V_i$ and $V_j$, where $1 \leq i \neq j \leq 3$ is covered by $E'_1$ is $\frac{|V_0|}{2}(n - 2|V_0| + 1)$. Thus, the number of uncovered pair of vertices between $V_i$ and $V_j$ for $1 \leq i \neq j \leq 3$ is $(|V_1||V_2| + |V_2||V_3| + |V_3||V_1|) - \frac{|V_0|}{2}(n - 2|V_0| + 1)$. These pairs should be covered by $E'_3$. Since each edges in $E'_3$ covers two pairs, this implies that $(|V_1||V_2| + |V_2||V_3| + |V_3||V_1|) - \frac{|V_0|}{2}(n - 2|V_0| + 1)$ should be even. However, since $|V_i|$ are odd for all $1\leq i \leq 3$, and $|V_0|$ is even, and $n - 2|V_0| + 1$ is even, the number is odd, a contradiction. Thus, $H$ has no $STS$.
\end{proof}

By taking each hypergraph to be the same, we can see that the value $\frac{3}{4}$ cannot be replaced with any smaller number also in the case of the transversal $STS$. Indeed, we conjecture that any $3$-uniform hypergraph systems with each hypergraph has minimum codegree at least $\frac{3}{4}n + C$ contains a transversal $STS$ for some universal constant $C$.

In order to get an explicit bound for $\delta_f^*$, we obtain a uniform upper bound for the function $\delta_f(\cdot)$.

\begin{theorem}\label{thm:threshold-uniformbound}
    $\delta_f(\ve) \leq \frac{3 + \sqrt{57}}{12} < 0.88$ for all $\ve\in [0, 1)$.
\end{theorem}

Together with \Cref{thm:main}, this theorem immediately implies the following.

\begin{corollary}\label{cor:0.88}
    For any positive real number $\ve > 0$, there is $n_0 = n_0(\ve)$ that satisfies the following for all $n$ satisfying $n \equiv 1 \text{ or } 3 \pmod{6}$.
    Let $\calH = \{H_1, \dots, H_{\frac{n(n-1)}{6}}\}$ be the collection of $3$-uniform hypergraph on a common $n$-vertex set $V$ such that the minimum codegree of $H_i$ is at least $\left(\frac{3 + \sqrt{57}}{12} + \ve \right) n$ for each $i\in [\frac{n(n-1)}{6}]$. Then there is a $\calH$-transversal $STS$ on $V$.
\end{corollary}

We remark that \Cref{thm:simple-main} can be directly deduced from \Cref{cor:0.88} by taking all hypergraphs $H_1, \cdots, H_{\frac{n(n-1)}{6}}$ the same.

In our proof, we develop a novel way of getting an approximated structure via fractional hypergraph matchings. We also utilize the iterative absorption method for both edge and color absorptions.
Iterative absorption was first introduced in~\cite{kuhn2013hamilton} and the method has significant applications to several central problems in extremal combinatorics. (For instances, see~\cite{kuhn2013hamilton,knox2015edge,barber2016edge,barber2020minimalist,barber2017clique,kang2023thresholds,kwan2022high}.) The purpose of the iterative absorption method is to restrict the size of the uncovered structure to easily absorb.

The method of iterative absorption typically consists of three steps. First, we fix an appropriate sequence of the vertex subsets of the host structure that looks like a  ``vortex". The size of each vertex set in the vortex decreases rapidly and the final vertex set has only constant size. 
Second, we iteratively construct a desired structure based on the vortex. This step can be achieved by the so-called ``cover-down lemma", which is the main technical hurdle for the iterative absorption. By the cover-down lemma, we construct a partial desired structure that ``covers" all (hyper)edges outside a smaller set in the vortex and reduces the problem to a smaller problem inside this smaller set. Repeatedly applying the cover-down lemma, we can ensure that the remaining (hyper)edges are all in the final constant-size vertex set in the vortex.
Lastly, we ``absorb" these remaining (hyper)edges by using so-called exclusive absorbers that can absorb small structures. The exclusive absorber is a flexible construction that forms a desired construction together with the ``left-over" in the final vertex set of the vortex no matter what the left-over structure is. We construct such exclusive absorbers and embed them prior to reaching the final vertex set of the vortex. 

\paragraph{Proof sketch.}
Our proof strategy follows the aforementioned scheme. We first fix a vortex for the transversal $STS$. Since we are concerned with a transversal structure, we extend the definition of vortex into ``color-vortex" and ``transversal-vortex" (See \Cref{def:color-vortex,def:transversal-vortex}) which satisfies certain structural properties of all hypergraphs with respect to an appropriately chosen color set in each level of vortex. 

We need a cover-down lemma with respect to the given transversal-vortex. In other words, we need to be able to cover all edges not inside a smaller set in the vortex as a shadow of a linear $3$-graph in such a way that the hyperedges of the $3$-graph come from different colors. Forgetting the colors for the moment, the approximate construction of such $3$-graph can be reduced to a hypergraph matching problem. Indeed, for the given set $E\subseteq \binom{V(H)}{2}$ of all the edges to be covered and the given $3$-graph $H$ of available hyperedges, consider an auxiliary graph $H'$ on the vertex set $E$ with edge set $\{\{xy, yz, zx\}\subseteq E: xyz \in E(H)\}$. Then an almost perfect matching in $H'$ is an approximate linear $3$-graph that we seek. 

There are extensive studies about conditions guaranteeing an almost perfect matching in hypergraph. Arguably, Pippenger-Spencer theorem~\cite{pippenger1989asymptotic} is a most useful theorem utilizing R\"{o}dl nibble method to find an almost perfect matching. It states that any sufficiently regular hypergraph with a small maximum codegree has an almost perfect matching covering all but $\ve$-fraction of vertices. We wish to apply this theorem to find an almost perfect matching in the auxiliary hypergraph $H'$. There are two main obstacles to this plan. One is that we must ensure that certain vertices of $H'$ must be covered by a matching. The ability to cover all of the predesignated vertices in $H'$ is indeed necessary. First, we need to cover all edges not inside the smaller set in the vortex. Second, for the later color-absorption, it is crucial that the last left-over colors are only the colors inside the final color set in the color-vortex. Hence, we need to prove that we can find an almost perfect matching in $H'$  covering all vertices from a predesignated set. Our plan for overcoming this issue is the following. Divide the vertex set (and color set) into two parts so that ``predesignated" vertices and colors belong to the first set. After finding an almost perfect matching of $H'$ in the first set, if only a very small number of predesignated vertices and colors have remained, once very few of them are remaining, we can greedily cover them using a hyperedge of $H'$ between the two parts. However, it turns out that the leftover has to be very small (smaller than a certain polynomial fraction of $H'$.) As  Pippenger-Spencer theorem only guarantees the left-over size is less than $\ve$ fraction of vertices, it is not sufficient for our purpose. We instead use a quantitatively stronger theorem from~\cite{kang2020new}. In fact, as we need a much stronger result than what~\cite{kang2020new} offers, we track nine consecutive color sets (instead of two) of color- and transversal-vortex and repeat the above process multiple times to make sure the number of remaining predesignated vertices and colors are smaller than what we need and the greedy covering does not ruin the remaining hypergraph too much. 

Another obstacle is that the auxiliary hypergraph $H'$ can be far from regular. (The codegree condition is fine as $H'$ has maximum codegree $1$.) Moreover, to utilize the quantitatively stronger hypergraph matching theorem mentioned above, we need a very strong regularity. In order to obtain a very regular subgraph of $H'$, we utilize the fractional matching of $H$. As we assume that the minimum codegree of $H$ is at least $(\delta^*_f + \ve)n$, the auxiliary graph $H'$ has a fractional perfect matching in a robust way. We repeatedly take a fractional perfect matching and delete hyperedges with large weights, we obtain many fractional perfect matching in $H'$. Taking a union of them yields a perfectly regular edge-weighted hypergraph with bounded weights. By discretizing the weights, we can find a very regular subgraph of the hypergraph $H'$ which is a multi-hypergraph with bounded multiplicity. For such a multi-hypergraph, we can apply the hypergraph matching theorem to obtain an almost perfect matching as desired.

In the preceding paragraphs, we often ignored the colors and only discussed the vertices. To show that we can incorporate the colors in our plan, consider the hypergraph obtained by adding new vertices that represent each color to $H'$ and replace the edges by its union with the new vertex corresponding to its color. Then a perfect matching in this new hypergraph is exactly the transversal that we are looking for. Although there are some more technical difficulties, this gives some idea on how to incorporate the colors in the cover-down lemma.

For the absorption step, we need to construct an exclusive absorber that can absorb the colors as well. To achieve this, we introduce a variant of the notion of degeneracy, namely ``edge-degeneracy" (See \Cref{def:edge-degeneracy}). A small edge-degeneracy ensures that it can be greedily embedded in a dense $3$-graph. Hence, we construct a desired exclusive absorber with small edge-degeneracy, and we can greedily embed them for the last vertex set of transversal-vortex using the small edge-degeneracy. Once we obtain the above cover-down lemma (\Cref{lem:cover-down}) and the exclusive absorber, we can follow the scheme of iterative absorption to finish the proof.

\paragraph{Organization.}
The rest of the paper is organized as follows. In \Cref{sec:prelim}, we declare our notations and introduce several tools for the proof of main theorems. In \Cref{sec:fractional-threshold}, we prove \Cref{thm:threshold-uniformbound} by using the max-flow min-cut theorem. We show that we can obtain the approximated transversal $STS$ via fractional $STS$ in \Cref{sec:pseudorandomSTS+absorber}. We also construct the exclusive absorbers in \Cref{sec:pseudorandomSTS+absorber} for the final absorption step. In \Cref{sec:cover down}, we state and prove the cover-down lemma for transversal $STS$ which is the main lemma to prove our main result. By combining all of them, in \Cref{sec:proof}, we prove \Cref{thm:main} by using the iterative absorption method. Finally, we end with concluding remarks and several open problems in \Cref{sec:concluding}.


\section{Preliminaries}\label{sec:prelim}

\subsection{General notations}\label{subsec:notations}

We treat large numbers as if they are integers by omitting floors and ceilings.
When we claim a statement holds if $0 < \alpha \ll \beta_1, \beta_2 \ll \gamma < 1 $, it means that there exist functions $f$ and $g$ such that the statement holds whenever $\beta_1 ,\beta_2 < f(\gamma)$ and $\alpha < g(\beta_1, \beta_2)$. We will not compute those functions explicitly.
We denote the set $\{1, 2, \dots, n\}$ by $[n]$.
For a hypergraph $H$, $V(H)$ and $E(H)$ denote the set of vertices and the set of edges of $H$, respectively. For a positive integer $k$, we write $k$-uniform hypergraph as $k$-graph. We simply write $2$-graph as graph. We write $H_{\mathrm{simp}}$ for a \emph{simplification} of $H$, which is a hypergraph obtained by identifying all multiple edges of $H$ into a single edge. For a graph $G$ and disjoint vertex sets $X, Y \subseteq V(G)$, $G[X, Y]$ denotes a subgraph of $G$ with the vertex set $X\cup Y$ and the edge set $\{xy\in E(G): x\in X, y\in Y\}$.

For a $k$-graph $H$, we denote the $i$-degree of $H$ as follows, where $1 \leq i \leq k-1$. For a set $X\in \binom{V(H)}{i}$, we write \[d_i(X) \defeq |\{f\in E(H): X\subset f\}|.\] We denote by $\Delta_i(H)$ and $\delta_i(H)$ the maximum and minimum $i$-degree of $H$, respectively. We simply write $\Delta(H)$ and $\delta(H)$ for $\Delta_1(H)$ and $\delta_{1}(H)$, respectively. For given $U\subseteq V(G)$ and $f\in \binom{V(G)}{k-1}$, we define \[N_H(f; U) \defeq \{v\in U: f\cup \{v\}\in E(H)\}\] and \[d_H(f; U) \defeq |N_H(f; U)|.\] If $H$ is clear from the context, we omit the subscript $H$.
For a $3$-graph $H$, we say that $H$ has \emph{essential minimum codegree} at least $\delta$ if $d_2(e) \geq \delta$ for all $e\in E(H^{(2)})$ and we write $\delta_{\mathrm{ess}}(H) \geq \delta$.

Let $H$ and $H'$ be two hypergraphs. Then we denote $H\setminus H'$ the hypergraph with the vertex set $V(H)$ and the edge set $E(H)\setminus E(H')$. Similarly, $H\cup H'$ is a hypergraph on the vertex set $V(H)\cup V(H')$ with the edge set $E(H)\cup E(H')$.

We say that a $3$-graph $H$ is a $K_3$-decomposition of a graph $G$ if $H$ is a linear hypergraph and $H^{(2)} = G$.

\subsection{Pseudorandom hypergraphs}
For non-negative real numbers $a, b$ and $c$, the expression $x = (a\pm b)c$ means that $x$ is between $(a - b)c$ and $(a + b)c$. If $c = 1$, we write $(a\pm b)$ instead of $(a\pm b)c$. We say a hypergraph $H$ is \emph{$(a\pm b)c$-regular} if $(a - b)c \leq \delta(H) \leq \Delta(H) \leq (a + b)c$.

\begin{definition}
    Let $D, \delta, \tau > 0$ be positive real numbers. For a multi-hypergraph $H$, we say $H$ that is \emph{$(D, \tau, \delta)$-pseudorandom} if $H$ is $(1 \pm \tau )D$-regular and $\Delta_2(H) \leq \delta D$.
\end{definition}

\begin{definition}
    Suppose that $M$ is a matching in a hypergraph $H$ and $\calF \subset 2^{V(H)}$ is a collection of subsets of vertices, and $\ve \geq 0$ is a non-negative real number. We say $M$ is an $(\calF, \ve)$-pseudorandom matching if the size of the set $F\setminus M$ is at most $\ve |F|$ for every $F\in \calF$.  
\end{definition}

Let $H$ be a simple hypergraph. We call a function $\psi: E(H) \to [0, 1]$ \emph{weighted} on $H$. For a vertex $v\in V(H)$, we define $d^{\psi}(v) \defeq \sum_{v\in e \in E(H)} \psi(e)$, a weighted degree of $v$. Similarly, for a pair of distinct vertices $u, v$, we write $d_2^{\psi}(\{u, v\}) \defeq \sum_{\{u, v\}\subseteq e \in E(H)} \psi(e)$, a weighted codegree of $\{u, v\}$.
We denote by $\Delta^{\psi}(H)$, $\delta^{\psi}(H)$, $\Delta_2^{\psi}(H)$, and $\delta_2^{\psi}(H)$ be the maximum/minimum weighted (co)degree, respectively. We denote by $\lVert \psi \rVert$ the maximum value of $\psi$, that is $\max\{\psi(e): e\in E(H)\}$.

If a weighting $\psi$ on $H$ satisfies $d^{\psi}(v) \leq 1$ for every $v\in V(H)$, then we say that $\psi$ is a \emph{fractional matching} of $H$.
We say $\psi$ is a \emph{perfect fractional matching} of $H$ if $d^{\psi}(v) = 1$ holds for all $v\in V(H)$.

\begin{definition}
    Let $H$ be a hypergraph.
    For non-negative real number $d, \tau, \delta$, we say a fractional matching $\psi$ of $H$ is \emph{$(d, \tau, \delta)$-pseudorandom fractional matching} if for all $v\in V(H)$, we have $d^{\psi}(v) = (1 \pm \tau)d$ and $\Delta_2^{\psi}(H) \leq \delta d$.
\end{definition}

Let $H$ be a $3$-graph. We note that a weighted subgraph $\phi$ of $H$ is a fractional $STS$ if $\sum_{e\subset f\in E(H)}\phi(f) \leq 1$ holds for all $e\in E(H^{(2)})$. Remark that if $\sum_{e\subset f\in E(H)}\phi(f) = 1$ holds for all $e\in E(H^{(2)})$, then we call $\phi$ is a perfect fractional $STS$ of $H$.

\begin{definition}
    Let $H$ be a $3$-graph. For non-negative real numbers $d, \tau, \delta$, we say that a fractional $STS$ $\phi$ on $H$ is \emph{$(d, \tau, \delta)$-pseudorandom fractional $STS$} if for all $e\in E(H^{(2)})$, we have $d^{\phi}(e) = (1 \pm \tau) d$ and $\lVert \phi \rVert \leq \delta d$.
\end{definition}

Let $H$ be a $ 3$-uniform simple hypergraph. We denote by $H_{aux}$ the auxiliary $3$-graph $F$, where $V(F) = E(H^{(2)})$ and $E(F) = \{\{e_1, e_2, e_3\}\subset E(H^{(2)}): \exists f\in E(F) \text{ such that } f = e_1 \cup e_2 \cup e_3\}$. We note that $H_{aux}$ is a linear hypergraph.

The following observation is immediate from the definition of $H_{aux}$.

\begin{observation}\label{obs:correspondence}
    Let $H$ be a $3$-uniform simple hypergraph and $d, \tau, \delta$ be non-negative real numbers. Then $H$ has a $(d, \tau, \delta)$-pseudorandom fractional $STS$ of $H$ if and only if $H_{aux}$ has a $(d, \tau, \delta)$-pseudorandom fractional matching.
\end{observation}


\subsection{Probabilistic inequalities}\label{subsec:prob-inequality}

\begin{lemma}[Chernoff bound~\cite{janson2011random}]\label{lem:chernoff}
    Let $X_1, \dots, X_n$ be $n$ independent Bernoulli random variables and let $X = \sum_{i\in [n]} X_i$. Then the following inequality holds for every $\delta \in [0, 1)$.

    $$Pr\left(|X - \mathbb{E}[X]| \geq \delta \mathbb{E}[X] \right) \leq 2 exp\left(-\frac{\delta^2 (\mathbb{E}[X])^2}{3}\right).$$
    
\end{lemma}

Similar inequality also holds for hypergeometric distributions.

\begin{lemma}[Chernoff bound for hypergeometric distribution~\cite{janson2011random}]\label{lem:chernoff-hypergeometric}
    Let $X$ be a random variable that follows a hypergeometric distribution with parameters $N, K, n$. Then the following tail bound holds for every $\delta \in [0, \frac{3}{2}]$.

    $$Pr\left(|X - \mathbb{E}[X]| \geq \delta \mathbb{E}[X] \right) \leq 2 exp\left(-\frac{\delta^2 
 \mathbb{E}[X]}{3}\right).$$
\end{lemma}


\section{Fractional thresholds}\label{sec:fractional-threshold}

In this section, we prove \Cref{thm:threshold-uniformbound}.
The proof utilizes the techniques of Dross~\cite{dross2016fractional} developed to investigate the threshold of the fractional triangle decomposition of the graph.

\begin{proof}[Proof of \Cref{thm:threshold-uniformbound}]
    Let $H$ be a $3$-graph on the vertex set $V$ and let $G = H^{(2)}$. 
    For a copy $K$ of $K^{(3)}_4$ in $H$ on a vertex set $\{u_1, u_2, v_1, v_2\}$, we call a tuple $(K, \{u_1v_1, u_2v_2\})$ a \emph{rooted $K^{(3)}_4$ on $u_1v_1$ and $u_2v_2$.} Note that a copy of $K^{(3)}_4$ can be rooted in three different ways.

    Suppose $(K, \{u_1v_1, u_2v_2\})$ is a rooted $K^{(3)}_4$ in $H$ and $\eta$ is a positive real number, and $\psi : E(H) \to \mathbb{R}$ is a function. Increase each value of $\psi(\{u_1,v_1, u_2\})$ and $\psi(\{u_1, v_1, v_2\})$ by $\frac{\eta}{2}$ and decrease each value of $\psi(\{u_1, u_2, v_2\})$ and $\psi(\{v_1, u_2, v_2\})$ by $\frac{\eta}{2}$. We call this operation as \emph{$\eta$-flow-opertion} on $K$.
    By the $\eta$-flow-operation on $K$, we obtain a function $\psi': E(H) \to \mathbb{R}$ such that the following holds.

    \begin{enumerate}
            \item[$\bullet$] $\psi'(f) = \psi(f)$ for all $f\in E(H) \setminus E(K)$,
            \item[$\bullet$] $\sum_{e \subseteq f\in E(H)}\psi'(f) = \sum_{e \subseteq f\in E(H)} \psi(f)$ for all $e\in E(G)\setminus \{u_1v_1, u_2v_2\}$,
            \item[$\bullet$] $\sum_{\{u_1, v_1\}\subseteq f\in E(H)}\psi'(f) = \sum_{\{u_1, v_1\}\subseteq f\in E(H)} \psi(f) + \eta$, and $\sum_{\{u_2, v_2\}\subseteq f\in E(H)}\psi'(f) = \sum_{\{u_2, v_2\}\subseteq f\in E(H)}\psi(f) - \eta$. 
    \end{enumerate}

    Let $\delta = \frac{3 + \sqrt{57}}{12}$.
    We now assume the $3$-graph $H$ satisfies $\delta_{\mathrm{ess}}(H) \geq \delta n$. We may assume that for every edge ${v_1, v_2, v_3}$ of $H$, at least one of the pairs ${v_1v_2, v_2v_3, v_3v_1}$ has codegree exactly $\delta n$. If this is not the case, we simply remove the edge ${v_1, v_2, v_3}$ from $H$ to reduce the excess codegrees. Then all edges in $H$ are contained in at most $\delta n$ copies of $K^{(3)}_4$, hence contained in at most $3\delta n$ copies of rooted $K^{(3)}_4$.
    
    We now consider a constant weight function $\psi_0$ satisfying $\psi_0(e) = \omega$ for all $e\in E(H)$ where $\omega = \frac{|E(G)|}{3 |E(H)|}$.
    As we observed, we note that every edge in $H$ is contained in at most $3\delta n$ copies of rooted $K^{(3)}_4$. By the $\eta$-flow-operation, for every $3$-uniform edge $f$, if we take $\eta \leq \frac{2\omega}{3\delta n}$, then the final weights of all $3$-uniform edges of $H$ be nonnegative. 
    For every $e \in E(G)$, let $d_H(e)$ be the number of $3$-uniform edges of $H$ containing $e$. Consider an auxiliary graph $F$ on the vertex set $V(F) = E(G) \cup \{s, t\}$ where $s, t$ are distinct new vertices. 
    Let
    \begin{align*}
        E_1 &\defeq \{se: e\in E(G): d_H(e)  > \frac{1}{\omega}\},\\
        E_2 &\defeq \{et: e\in E(G): d_H(e) \leq \frac{1}{\omega}\},\\
        E_3 &\defeq \{ee': e, e'\in E(G), \text{ there is a rooted $K^{(3)}_4$ on $e$ and $e'$ in $H$ }\}.
    \end{align*}
    
    Let $E(F) = E_1 \cup E_2 \cup E_3$. We say that two vertices $s$ and $t$ as source and sink, respectively. For each edge $ee'$ in $E_3$, let the capacity of $ee'$ be $\frac{2w}{3\delta n}$. For each $se\in E_1$ and $e't\in E_2$, let the capacity of $se$ and $e't$ be $d_H(e) w - 1$ and $1 - d_H(e') w$, respectively.

    Let $M = \sum_{e\in E(G): se\in E_1} (d_H(e) \omega - 1)$. If we find a flow of value $M$ from $s$ to $t$, then it implies that there is a perfect fractional $STS$ of $H$ by applying $\eta$-flow-operations depending on the flow, adjusting the parameter $\eta$. Assume $F$ does not allow a flow of value $M$. Then by the max-flow min-cut theorem, there is a vertex partition $(X, Y)$ of $V(F)$ such that $s\in X$, $t\in Y$, and the sum of capacities of edges between $X$ and $Y$ is less than $M$. Let $X_0 = X\setminus \{s\}$ and $Y_0 = Y\setminus \{t\}$ and let $|X_0| = \ell$. Then $|Y_0| = |E(G)| - \ell$. Let $N_X$ be the average of $d_H(e)$ where $e$ is in $X_0$ and let $N_Y$ be the average of $d_H(e)$ where $e$ is in $Y_0$.

    We note that for every $e\in E(G)$, there are at least $\frac{d_H(e)(d_H(e) - 2(1 - \delta)n)}{2}$ edges $e'$ such that there exists rooted $K^{(3)}_4$ on $e$ and $e'$. Thus $F[X, Y]$ contains at least $\sum_{e\in X_0} \left( \frac{d_H(e)(d_H(e) - 2(1 - \delta)n)}{2} - \ell \right)$ edges of $E_3$. Thus, we have the following inequality.
    \begin{equation}\label{eq:1}
        \sum_{e\in X_0} \left( \frac{d_H(e)(d_H(e) - 2(1 - \delta)n)}{2} - \ell \right) \frac{2\omega}{3\delta n} + \sum_{e\in Y_0: se\in E_1} (d_H(e) w - 1) - \sum_{e\in X_0: et\in E_2} (d_H(e) \omega - 1) < M.
    \end{equation}
    Since $M = \sum_{e\in E(G): se\in E_1} (d_H(e) \omega - 1) = \sum_{e\in E(G): et\in E_2} (1 - d_H(e) \omega)$, from \eqref{eq:1}, we obtain the following inequality.

    \begin{equation*}
        \frac{2w}{3\delta n}\sum_{e\in X_0} d_H(e)(d_H(e) - 2(1 - \delta)n) - 2 \ell) - 2\sum_{e\in X_0} (d_H(e) \omega - 1) < 0.
    \end{equation*}
    
    As the convexity of the function $f(x) = x^2$ implies $N_X \leq \frac{1}{n}\sum_{e\in X_0} d_H(e)^2$, we have the following.
    \begin{equation}\label{eq:2}
        \frac{2\omega}{3\delta n} (N_X(N_X - 2(1 - \delta)n) - 2\ell) - 2(N_X \omega - 1) < 0.
    \end{equation}
    
    With the same reason on the $Y_0$ side, the following inequality also holds.
    \begin{equation}\label{eq:3}
        \frac{2\omega}{3\delta n} (N_Y(N_Y - 2(1 - \delta)n) - 2(|E(G)| - \ell)) - 2(1 - N_Y \omega) < 0.
    \end{equation}

    Then \eqref{eq:2} and \eqref{eq:3} imply the following two inequalities, respectively.
    \begin{equation}\label{eq:4}
        N_X(N_X - 2(1 - \delta)n) + \frac{3\delta n}{\omega} (1 - N_X \omega) < 2\ell.
    \end{equation}
    
    \begin{equation}\label{eq:5}
        2\ell < 2|E(G)| - N_Y(N_Y - 2(1 - \delta)n) + \frac{3\delta n}{\omega}(1 - N_Y \omega).
    \end{equation}

    Since $2|E(G)| < n^2$, from \eqref{eq:4} and \eqref{eq:5}, we obtain the following.
    \begin{equation}\label{eq:6}
        N_X(N_X - (2 + \delta)n) < n^2 - N_Y(N_Y -(2 - 5\delta)n)
    \end{equation}

    Since every edge $e\in E(G)$ is contained in at least $\delta n$ hyperedges of $H$, we have $\delta n \leq N_X, N_Y \leq n$. In order to deduce an inequality only cares about $\delta$, we observe that the left-hand side of \eqref{eq:6} is minimized when $N_X = n$ and the right-hand side of \eqref{eq:6} is maximized when $N_Y = \delta n$. Thus, we deduce the following inequality.

    \begin{equation}\label{eq:7}
        -(1 + \delta) < 1 - \delta^2 + \delta(2 - 5\delta).
    \end{equation}

    By solving \eqref{eq:7}, we have $\delta < \frac{3 + \sqrt{57}}{12}$, a contradiction.
\end{proof}


\section{Pseudorandom structures and exclusive absorbers}\label{sec:pseudorandomSTS+absorber}

In this section, we state and prove lemmas useful for finding approximated structures and executing the final absorption step. Since we plan to iteratively construct an approximated transversal $STS$ for each set in the vortex in each step, we need to show we can find approximated transversal $STS$ with respect to the vortex (\Cref{def:transversal-vortex}) structure. We can achieve this using the notion of pseudorandomness. In the following subsection, we show how we obtain a pseudorandom fractional $STS$.

\subsection{Pseudorandom fractional STS}\label{subsec:fractional matching}

The goal of this section is to prove the following lemma.

\begin{lemma}\label{lem:pseudorandom-STS}
    For every real number $\ve > 0$, there is $n_0 = n_0(\ve)$ such that the following holds for all $n\geq n_0$.
    Let $H$ be an $n$-vertex $3$-uniform simple hypergraph with $\delta_{\mathrm{ess}}(H) \geq (\delta_f(\ve) + \ve)n$ and $\delta(H^{(2)}) \geq (1 - \ve)n$.
    Then there is $\left(1, 0, \frac{\log^2 n}{n}\right)$-pseudorandom fractional $STS$ of $H$.
\end{lemma}

To prove \Cref{lem:pseudorandom-STS}, we utilize pseudorandom fractional matchings.
The following lemma demonstrates how to obtain a pseudorandom fractional matching by using many perfect fractional matching.

\begin{lemma}\label{lem:frac to int matching}
    Let $C$ and $D$ be positive integers and $\gamma, \delta, \eta$ be real numbers in the interval $(0, 1)$ which satisfies the inequality $\delta D \left(\gamma + 2^{-\gamma C} \right) < \frac{\eta}{2}$. 
    Let $H$ be a multi-hypergraph with $\delta(H) \geq D$ and $\Delta_2(H) \leq \delta D$. Assume for every subgraph $F$ of $H$ with $\Delta(F) \leq C$, the sub-hypergraph $H \setminus F$ has a perfect fractional matching. Then $H$ has a $(1, 0, \eta)$-pseudorandom fractional matching $\psi$ with $\lVert \psi \rVert \leq \gamma + 2^{- \lfloor \gamma C \rfloor}$.
\end{lemma} 

The proof of the above lemma is based on the iterative application of the following observation.

\begin{observation}\label{obs:normalized fractional matching}
    Let $H$ be a hypergraph. Let $\psi_1, \dots, \psi_m$ be perfect fractional matchings of $H$. Then $\frac{1}{m}\sum_{i = 1}^m\psi_i$ is also a perfect fractional matching of $H$.
\end{observation}

\begin{proof}[Proof of \Cref{lem:frac to int matching}]
    
    Let $A$ be a hypergraph. We denote by $\calF_{A}$ the set of all perfect fractional matchings of $A$.
    We define $\omega(A) \defeq \min \{\lVert \psi \rVert : \psi \in \calF_{A}\}$ if $\calF_{A} \neq \emptyset$ and otherwise, $\omega(A) \defeq \infty$. We note that $\omega(A)$ is well-defined since it is defined by the minimum of a continuous function on the compact domain, so the minimum can be attained.
    For a positive integer $d$, we denote by $H(d)$ the set of sub-hypergraph $F$ of $H$ such that $\Delta(F) \leq d$. Then the following claim holds.

    \begin{claim}\label{clm:inductive-regularization}
        Let $d$ and $d'$ be two positive integers such that $d' \leq d$. If $\max \{\omega(H\setminus F) : F \in H(d)\} \leq \alpha$, then $\max \{\omega(H\setminus F') : F' \in H(d')\} \leq \gamma + \frac{\alpha - \gamma}{\lfloor (d - d')\gamma \rfloor + 1}$    
    \end{claim}

    \begin{claimproof}
        Let $m = \lfloor (d - d')\gamma \rfloor$. Fix an arbitrary hypergraph $F' \in H(d')$.
        We will iteratively choose suitable perfect fractional matchings $\psi_{0}, \psi_{1}, \dots, \psi_{m}$ from $H\setminus F'$. We initialize $F_0 = F'$ and $H_0 = H\setminus F_0$. Assume we have $H_i$ and $F_i$ such that $\Delta(F_i) \leq d$ where $0 \leq i \leq m$. Then $F_i \in H(d)$, the hypergraph $H_i$ has a perfect fractional matching $\varphi$ with $\lVert \varphi \rVert \leq \alpha$. We set such perfect fractional matching as $\psi_{i}$. We now update $F_{i+1} \defeq F_i \cup \{e\in E(H_i): \psi_{i}(e) > \gamma \}$ and define $H_{i + 1} = H\setminus F_{i+1}$.
        We note that since $\psi_i$ is a perfect fractional matching, for every vertex $v\in V(\calH)$, we have $d_{F_{i+1}}(v) \leq d_{F_{i}}(v) + \lfloor \frac{1}{\gamma} \rfloor$. Thus, $\Delta(F_{m}) \leq \Delta(F_0) + m\lfloor \frac{1}{\gamma} \rfloor \leq d' + m \lfloor \frac{1}{\gamma} \rfloor \leq d' + d - d' = d$. Thus, we obtain perfect fractional matchings $\psi_0, \dots, \psi_m$ of $H\setminus F'$. Let $\psi = \frac{1}{m+1} \sum_{i = 0}^m \psi_i$. Then by \Cref{obs:normalized fractional matching}, $\psi$ is also a perfect fractional matching of $H\setminus F'$. 
        
        We now claim that $\lVert \psi \rVert$ is at most $\gamma + \frac{\alpha - \gamma}{m+1}$. For an edge $e\in H\setminus F'$, denote by $w_e$ the vector $(\psi_0(e), \psi_1(e), \dots, \psi_m(e))$. If $e \notin H_i$, then we simply define $\psi_i(e) = 0$. For some $a \in \{0, 1, \dots, m\}$, if $\psi_a(e) > \gamma$, then $\psi_a(e) \leq \alpha$ and $\psi_i(e) \leq \gamma$ for every $i < a$ and $\psi_{j}(e) = 0$ for every $j > a$. Thus, we have $\psi(e) = \frac{1}{m+1}\sum_{i = 0}^m \psi_i(e) \geq \frac{1}{m+1}\left(m\gamma + \alpha \right) = \gamma + \frac{\alpha - \gamma}{m+1}$. Since $e$ is an arbitrary edge of $H\setminus F'$ and $F'$ is an arbitrary hypergraph in $H(d')$. This proves the claim.
    \end{claimproof}

    We now use \Cref{clm:inductive-regularization} iteratively to complete the proof. Let $t = \lfloor \gamma C \rfloor$. For every $i \in \{0, 1, \dots, t\}$, let $d_i = C - \frac{i}{\gamma}$ and $\alpha_i = \gamma + \frac{1 - \gamma}{2^i}$. We claim that $\max\{\omega(H \setminus F) : F \in H(d_i)\} \leq \alpha_i$ for every $i\in \{0, 1, \dots, t\}$. If $i = 0$, then by the condition of $H$, for every sub-hypergraph $F \in H(d_0)$, where $d_0 = C$, $H\setminus F$ has a perfect fractional matching. Thus, $\max\{\omega(H \setminus F) : F \in H(d_0)\} \leq 1 = \alpha_0$. We now use induction on $i$. We apply \Cref{clm:inductive-regularization} where $d_{i-1}$, $d_i$, and $\alpha_{i-1}$ play role of $d$, $d'$, and $\alpha$, respectively. Then $\max\{\omega(H \setminus F) : F \in H(d_i)\} \leq \gamma + \frac{\alpha_{i-1} - \gamma}{\lfloor (d_{i-1} - d_i)\gamma \rfloor + 1} = \gamma + \frac{1}{2}\left(\gamma + \frac{1-\gamma}{2^{i-1}} - \gamma \right) = \gamma + \frac{1 - \gamma}{2^i} = \alpha_i$. Thus, by induction, for every $i\in \{0, 1, \dots, t\}$, we have $\max\{\omega(H \setminus F) : F \in H(d_i)\} \leq \alpha_i$. This implies that $\omega(H) = \max \{\omega(H\setminus F) : F \in H(0)\} \leq \max \{\omega(H\setminus F) : F \in H(t)\} \leq \alpha_t = \gamma + \frac{1-\gamma}{2^t} \leq \gamma + \frac{1}{2^{\lfloor \gamma C \rfloor}}$.
    Thus, there is a perfect fractional matching $\psi$ in $H$, where $\lVert \psi \rVert \leq \gamma + \frac{1}{2^{\lfloor \gamma C \rfloor}}$.

    In order to finish the proof, it suffices to show that for every pair of distinct vertices $u, v \in V(H)$, the inequality $d_2^{\psi}(\{u, v\}) < \eta$ holds. Since $\Delta_2(H) \leq \delta D$, by our construction of $\psi$, we have $\sum_{\{u, v\}\subset e \in E(H)} \psi(e) \leq \delta D \left(\gamma + \frac{2}{2^{ \gamma C }} \right)$ for every pair of distinct vertices $u, v$. Since $2\delta D \left(\gamma + 2^{- \lfloor \gamma C \rfloor} \right) < \eta$, we conclude that $\Delta_2^{\psi}(H) < \eta$. This completes the proof.
\end{proof}

\begin{proof}[Proof of \Cref{lem:pseudorandom-STS}]
    Consider $H_{aux}$. Then $H_{aux}$ is a linear hypergraph and $\delta(H_{aux})$ is same as the minimum codegree of $e$ in $H$ where $e\in E(H^{(2)})$, thus we have $\delta(H_{aux}) \geq (\delta_f(\ve) + \ve)n$. We observe that there is an one-to-one correspondence between $E(H_{aux})$ and $E(H)$. 
    Moreover, if we remove $x$ many edges that are incident with $e\in V(H_{aux})$, then the codgree of $e$ in $H$ decreases exactly $x$. Thus, for every $3$-uniform sub-hypergraph $F$ in $H_{aux}$ with $\Delta(F) \leq \frac{\ve}{2}n$, all edges $e$ in the corresponding hypergraph $H'$ of $H_{aux} \setminus F$ has codegree at least $\left(\delta_f(\ve) + \frac{\ve}{2}\right)n$, where $H'_{aux} = H_{aux}\setminus F$. 
    By the definition of $\delta_f(\ve)$, there is a perfect fractional $STS$ in $H'$. Note that $H^{(2)} = H'^{(2)}$. By \Cref{obs:correspondence}, there is also perfect fractional matching in $H_{aux}\setminus F$. 
    Thus, we can apply \Cref{lem:frac to int matching} to $H_{aux}$ with parameters $\frac{\ve}{2}n, \delta_f(\ve) n, \frac{log^2 n}{3 n}, \frac{1}{\delta_f(\ve) n}, \frac{log^2 n}{n}$ playing the role with $C, D, \gamma, \delta, \eta$, respectively. By \Cref{lem:frac to int matching}, there is a $\left(1, 0, \frac{\log^2 n}{n}\right)$-pseudorandom fractional matching in $H_{aux}$. 
    Thus, by \Cref{obs:correspondence}, there is a $\left(1, 0, \frac{\log^2 n}{n} \right)$-pseudorandom $STS$ in $H$.
\end{proof}


\subsection{Exclusive absorbers}

In this section, we will prove that for every $K_3$-divisible graph $G$, there is a sufficiently sparse structure $A$ that can `absorb' $G$. That is, both $A^{(2)}$ and $A^{(2)}\setminus G$ have $K_3$-decomposition in $A$. We will take $G$ as a graph that is induced by uncovered edges by $3$-uniform edges in the final stage of the proof of \Cref{thm:main}. 

For a graph $G$ and a $3$-graph $H$, we say $H$ and $G$ are \emph{$K_3$-disjoint} if there is no hyperedge $e = \{v_1, v_2, v_3\}$ in $H$ such that $\{v_1v_2, v_2v_3, v_3v_1\} \subseteq E(G)$.

\begin{definition}
    Let $G$ be a graph. We say a $3$-graph $A$ is a \emph{$K_3$-absorber} for $G$ if $A$ and $G$ are $K_3$-disjoint, $A^{(2)}$ contains $G$ as an induced subgraph, and both $A^{(2)}$ and $A^{(2)} \setminus G$ have a $K_3$-decomposition in $A$.  
\end{definition}

The purpose of this section is to find a sparse absorber $A$ for every $K_3$-divisible graph. We use the following notion of edge-degeneracy to measure the sparseness.

\begin{definition}\label{def:edge-degeneracy}
    Let $H$ be a $3$-graph and let $F$ be an induced subgraph of $E(H^{(2)})$. The \emph{$F$-rooted edge-degeneracy} of $H$ is the smallest non-negative integer $d$ such that there is an ordering $v_1, v_2, \dots $, of $V(H) \setminus V(F)$ such that the following inequality holds for all $i$.

    $$|\{\{u, w\} \subseteq \left(V(F)\cup \{v_j : 1 \leq j < i\}\right) : \{v_i, u, w\}\in E(H) \}| \leq d.$$
\end{definition}

We now state the following lemma.

\begin{lemma}\label{lem:sparse-K3-absorber}
    For each $K_3$-divisible graph $G$, there exists a $K_3$-absorber $A$ for $G$ whose $G$-rooted edge-degeneracy is at most $4$.
\end{lemma}

In order to prove this, we define a specific structure called a transformer. Transformers will be useful to build a desired absorber. The idea of using transformers to construct exclusive absorbers was introduced in~\cite{barber2016edge,barber2020minimalist,glock2023existence}. Below is the definition of transformers.

\begin{definition}
    Let $S$ and $S'$ be two vertex-disjoint $K_3$-divisible graphs. We say a $3$-graph $T$ is a \emph{transformer of $(S, S')$} or \emph{$(S, S')$-transformer} if $T$ and $S\cup S'$ are $K_3$-disjoint, $T^{(2)}$ contains $S\cup S'$ as an induced subgraph, and both $T^{(2)}\setminus S$ and $T^{(2)}\setminus S'$ have a $K_3$-decomposition in $T$. 
\end{definition}

For two simple graphs $G_1$ and $G_2$, we write $G_1 \rightsquigarrow G_2$ if $|E(G_1)| = |E(G_2)|$ and there is a graph homomorphism $\phi: V(G_1) \to V(G_2)$ such that for every $e' \in E(G_2)$, there is an edge $e \in E(G_1)$ such that $\phi(e) = e'$. In other words, $\phi$ is an edge-bijective homomorphism.  

\begin{lemma}\label{lem:transformer-exist}
    For every pair of vertex-disjoint $K_3$-divisible graphs $(S, S')$ with $S\rightsquigarrow S'$, there is a transformer $T$ of $(S, S')$ which has the $(S\cup S')$-rooted edge-degeneracy at most $4$. 
\end{lemma}

\begin{proof}
    Let $\phi: V(S)\to V(S')$ be the graph homomorphism such that $\phi(E(S)) = E(S')$.
    Since $S$ is $K_3$-divisible, every vertex of $S$ has an even degree. Thus, $E(S)$ can be decomposed into cycles. Let $C_1, \dots, C_m$ be edge-disjoint cycle such that $S = \bigcup_{i=1}^m C_i$. Assume for every $i\in [m]$, there is a transformer $T_i$ for $(C_i, \phi(C_i))$ with $(C_i\cup \phi(C_i))$-rooted edge-degeneracy at most $4$. Then $\bigcup_{i = 1}^m T_i$ is a desired transformer. Thus, we may assume $S$ is a cycle $v_1, v_2, \dots, v_{\ell}$. We note that $E(\phi(S)) = \{\phi(v_1v_2), \dots, \phi(v_{\ell - 1}v_{\ell}), \phi(v_{\ell}v_1)\}$. We now define sets of edges as follows:
    \begin{align*}
        E_1& \defeq \{\{\phi(v_i), \phi(v_{i+1}), x_{i+1}\} : i\in [\ell]\}\\ 
        E_2& \defeq \{\{x_i, y_i, \phi(v_{i})\} : i\in [\ell]\}\\ 
        E_3& \defeq \{\{y_i, z_i, \phi(v_{i})\} : i \in [\ell]\}\\
        E_4& \defeq \{\{z_i, x_{i+1}, \phi(v_{i})\} : i \in [\ell]\}\\
        E_5& \defeq \{\{v_i, x_{i}, y_i\} : i \in [\ell]\}\\ 
        E_6& \defeq \{\{v_i, y_i, z_i\} : i \in [\ell]\}\\
        E_7& \defeq \{\{v_i, z_i, x_{i+1}\} : i \in [\ell]\}\\
        E_8& \defeq \{\{v_i, v_{i+1}, x_{i+1}\}\} : i \in [\ell]\}
    \end{align*}

Let $T$ be an $3$-graph with edges $\bigcup_{i\in [8]} E_i$. Then $T$ has $(S\cup \phi(S))$-rooted edge-degeneracy $4$ with the ordering $x_1, \dots, x_{\ell}, y_1, \dots, y_{\ell}, z_1, \dots, z_{\ell}$. For verifying edge-degeneracy $4$ with respect to this ordering, we build edges $E_1 \cup E_8$, $E_2\cup E_5$, and $E_3\cup E_4 \cup E_6 \cup E_7$ in turns. From this, each $X_i$ and $y_i$ has edge-degeneracy $2$ and each $z_i$ has edge-degeneracy $4$. Let $T_1 \defeq E_1 \cup E_3 \cup E_5 \cup E_7$ and $T_2 \defeq E_2 \cup E_4 \cup E_6 \cup E_8$. Then $T$ is a transformer of $(S, \phi(S))$, where $T^{(2)}\setminus S$ and $T^{(2)} \setminus \phi(S)$ have a $K_3$-decomposition $T_1$ and $T_2$, respectively.
\end{proof}

We are now ready to prove \Cref{lem:sparse-K3-absorber}.

\begin{proof}[Proof of \Cref{lem:sparse-K3-absorber}]
    Let $G$ be a $K_3$-divisible graph. Let $|E(G)| = 3m$ for some integer $m$. Let $F_G$ be a $3$-graph obtained from $G$ by introducing a new vertex $v_e$ for every edge $e\in E(G)$ and adding a hyperedge $e\cup \{v_e\}$ for each $e\in E(G)$. We denote by $G^*$ the graph on vertex set $V(F_G)$ with $E(G^*) = E(F_G^{(2)})\setminus E(G)$. Note that $G^*$ is the $1$-subdivision of $G$ and $G^{**}$ is the $3$-subdivision of $G$.
    Then we observe that $F = F_G \cup F_{G^*}$ be a $(G, G^{**})$-transformer which has $G$-rooted edge-degeneracy $1$. Note that $F^{(2)}_G$ and $F^{(2)}_{G^*}$ are $K_3$-decompositions of $F^{(2)}\setminus G^{**}$ and $F^{(2)} \setminus G$, respectively.   

    We denote by $S_{3m}$ the union of $3m$ copies of $C_4$ that share exactly one common vertex and all the other vertices are disjoint. We note that $G^{**}\rightsquigarrow S_{|E(G)|}$ by identifying the vertices in $V(G)$ as a single vertex. By \Cref{lem:transformer-exist}, there is a $3$-graph $T$ which is a $(G^{**}, S_{3m})$-transformer such that the $(G^{**}\cup S_{|E(G)|})$-rooted edge-degeneracy is at most $4$. This implies that there is a transformer of $(G, S_{|3m|})$ which has  $(G, S_{|3m|})$-edge-degeneracy at most $4$. 

    Let $M$ be the disjoint union of $3$-uniform edges where $V(M)$ is disjoint from $V(G)$ and let $M'$ be the graph $M^{(2)}$. Let consider $S_{3m}$ which is vertex disjoint from $V(G)$ and $V(M')$. Then there are two $3$-graph $T_1$ and $T_2$ such that $T_1$ and $T_2$ are $(G, S_{3m})$-transformer and $(S_{3m}, M')$-transformer, respectively, where $T_1$ has $(G\cup S_{3m})$-rooted edge-degeneracy is at most $4$ and $T_2$ has $(S_{3m} \cup M')$-rooted edge-degeneracy is at most $4$. 
    We obtain $A$ by placing $M$ and $S_{3m}$ first and then construct $T_1$ and $T_2$. This shows that the $G$-rooted edge-degeneracy of $A$ is at most $4$. Thus, it suffice to show that both $A^{(2)}$ and $A^{(2)}\setminus G$ has a $K_3$-decomposition in $A$. By the definition of transformer, $T_1^{(2)} \setminus S_{3m}$ and $T_2^{(2)} \setminus M'$ has a $K_3$-decomposition in $T_1$ and $T_2$, respectively. Since $M$ itself forms a $K_3$-decomposition of $M'$, this implies that $A^{(2)}$ have $K_3$-decompositions in $A$. Similarly, $T_1^{2} \setminus G$ and $T_2^{(2)}\setminus S_{3m}$ have $K_3$-decompositions in $T_1$ and $T_2$, respectively. This shows that $A^{(2)}\setminus G$ has a $K_3$-decomposition in $A$. Hence, $A$ is a $K_3$-absorber for $G$.
\end{proof}


\section{Transversal vortices and cover-down lemma}\label{sec:cover down}

\subsection{Transversal vortex}\label{subsec:transversal-vortex}

Since we will apply the iterative absorption method for finding a transversal, we define the color-vortex and the transversal-vortex as follows.

\begin{definition}[$(n, \ve)$-color-vortex]\label{def:color-vortex}
    A sequence $C_0 \supseteq C_1 \supseteq \cdots \supseteq C_{\ell}$ is an \emph{$(n, \ve)$-color-vortex} if $C_i \subseteq [n]$ and $|C_i| \leq \ve^{i-100}n$ for each $0 \leq i \leq \ell$.    
\end{definition}

For given two nonnegative integers $n$ and $m$, if we initialize $n_0 = n$ and inductively define $n_{i+1} \defeq  \lfloor \ve n_i \rfloor$, then we denote $\ell(n; m)$ by the minimum value of $\ell$ that satisfying $n_{\ell} \leq m$.

\begin{definition}[$(\alpha, \ve, m)$-transversal-vortex]\label{def:transversal-vortex}
    Let $\alpha, \ve$ be the positive real numbers and $m, n$ be integers such that $n \equiv 1 \text{ or } 3 \pmod{6}$. We write $\ell$ be the number $\ell(n; m)$ and let $\calC = \{C_0\supseteq C_1 \supseteq \cdots \supseteq C_{\ell-1}\}$ be a $(\frac{n(n-1)}{6}, \ve)$-color-vortex.
    Let $\calH = \{H_1, \dots, H_{\frac{n(n-1)}{6}}\}$ be a collection of $3$-graphs on common vertex set $V$ with size $n$. An \emph{$(\alpha, \ve, m)$-trasnversal-vortex} in $\calH$ with respect to $\calC$ is a sequence $U_0 \supseteq U_1 \cdots \supseteq U_{\ell}$ which satisfies the following.

    \begin{enumerate}
        \item[$(V1)$] $U_0 = V(H)$,
        \item[$(V2)$] $|U_{i}| =  \ve |U_{i-1}|$ for all $1 \leq i \leq \ell$,
        \item[$(V3)$] $|U_{\ell}| = m$,
        \item[$(V4)$] $d_{H_j}(e; U_{i+1}) \geq \alpha |U_{i+1}|$ for all $0 \leq i < \ell$, $j\in C_i$, and $e\in \binom{U_{i}}{2}$.
    \end{enumerate}
\end{definition}

The following lemma yields a desired transversal-vortex for a given collection of $3$-graphs.

\begin{lemma}\label{lem:vortex-transversalSTS}
    Let $0 < \frac{1}{n} \ll \frac{1}{m'} \ll \alpha, \ve < 1$. Here, $n$ and $m$ are integers and $n \equiv 1 \text{ or } 3 \pmod{6}$. Let $\calH = \{H_1, \dots, H_{\frac{n(n-1)}{6}}\}$ be a collection of $3$-graphs on common vertex set $V$ with size $n$. 
    Let $\alpha > 0$ be a real number. Assume for all $i\in [\frac{n(n-1)}{6}]$, we have $\delta_2(H_i) \geq \left(\alpha + \ve \right)n$. Let $\ell = \ell(n; m')$ and let $\calC = \{C_0 \supseteq C_1 \supseteq \cdots \supseteq C_{\ell-1}\}$ be a $(\frac{n(n-1)}{6}, \ve)$-color-vortex. 
    Then $\calH$ has a $\left(\alpha, \ve, m \right)$-transversal-vortex with respect to $\calC$ for some $\ve m' < m \leq m'$.
\end{lemma}

The proof of \Cref{lem:vortex-transversalSTS} is standard (we randomly take $U_{i}$ in $U_{i-1}$ for each $i\in [\ell]$, see~{\cite[Lemma 3.7]{barber2020minimalist}}), so we omit it.

One of the main ideas to prove \Cref{thm:main} is to keep tracking the remaining uncovered colors in the previous cover-down step. \Cref{lem:pseudorandom-STS} means that at each step we can obtain a subgraph in the remaining system which is highly pseudorandom. Then at each step, we can exponentially use the unused colors in the previous cover-down step which were expected to be used in the previous step. By this argument, after several steps, the number of unused color set in the certain stage is small enough to use them greedily while losing not so much pseudorandomness of the host structure. This idea will be formally stated and used in the following section.

Since, we need to track assigned colors in the few steps before, in \Cref{sec:proof}, we set the size of each color set in the color-vortex a bit bigger than the desired number of each color. However, to obtain a transversal-vortex, we cannot assign too many colors in each color set in the color-vortex. That is the reason that there is a size condition on the color set in the definition of the color-vortex.

\subsection{Cover-down lemma}\label{subsec:cover-down}

The main tool for the proof of \Cref{thm:main} is the following cover-down lemma for transversal $STS$. 

\begin{lemma}[Cover-down lemma]\label{lem:cover-down}
    Let $0 < \frac{1}{n} \ll \ve \ll 1$.
    Let $V$ be a set of $n$ vertices and let $U$ is a subset of $V$ with size $\ve n$. Let $U'\subseteq U$ be the set of size at most $\ve |U|$. Let $\calH = \{H_1, H_2, \dots, H_N\}$ is a collection of $3$-graph on the vertex set $V$ such that $\delta_{\mathrm{ess}}(H_i) \geq (\max\{\frac{3}{4}, \delta_f(3\ve)\}+10\ve)n$ for every $i\in [N]$, where $N$ is an integer such that $\ve^4 n^2 \leq N - \frac{|E(G\setminus G[U])|}{3} \leq \frac{\ve^3 n^2}{2}$. 
    Moreover, $H_i^{(2)}$ is the same graph $G$ for every $i\in [N]$.
    Suppose that $G$ is a $K_3$-divisible graph on $V$ with $\delta(G) \geq (1 - \ve)n$, $d_G(v; U) \geq (1 - 2\ve)|U|$ for all $v\in V$, and $G[U]$ is a complete graph.
    We further assume that for every $e\in E(G)$, the inequality $d_{H_i}(e;U) \geq (\max\{\frac{3}{4}, \delta_f(3\ve)\}+9\ve)|U|$ holds for all $i\in [N]$.
    Let $X_1, X_2, \dots, X_9$ be a partition of $[N]$ such that $|X_i| \leq \ve n^{i/4}$ for each $i\in [8]$. Then there exist $W\subseteq [N]$ and a $3$-graph $T$ on $V$ with $|E(T)| = |W|$ and $\delta_2(T) \leq 1$ that satisfies the following.

    \begin{enumerate}
        \item[$(D1)$] $X_1 \subseteq W$, $|X_i \setminus W| \leq \ve (\ve n)^{(i-1)/4}$ for each $i\in [9]\setminus \{1\}$,
        \item[$(D2)$] $(G \setminus G[U]) \subseteq T^{(2)} \subseteq (G \setminus G[U'])$,
        \item[$(D3)$] $\delta((G\setminus T^{(2)})[U]) \geq (1 - \ve)|U|$,
        \item[$(D4)$] $d_{(G\setminus T^{(2)})[U]}(u; U') \geq (1 - 2\ve)|U'|$ for all $u\in U$,
        \item[$(D5)$] there exists a bijective function $\phi: E(T) \to W$ such that $e\in E(H_{\phi(e)})$ for all $e\in E(T)$.
    \end{enumerate}
\end{lemma}

\Cref{lem:cover-down} means that we can iteratively construct the desired structure and it allows us to focus on the smaller set $U'$ after applying \Cref{lem:cover-down}. Moreover, in some sense, we can absorb a small set of colors to complete the transversal structure.

The following lemma shows that we can build a pseudorandom hypergraph based on given pseudorandom fractional matching.

\begin{lemma}\label{lem:matching-to-hypergraph}
    Let $k$ and $n$ be an integer and $D, d, \tau,\delta$ be non-negative real numbers. Let $H$ be a $k$-uniform simple hypergraph that has a $(d, \tau, \delta)$-pseudorandom fractional matching and $\Delta(H) \leq \tau D$. Then there is a multi-hypergraph $F$ such that $F_{\mathrm{simp}}$ is a sub-hypergraph of $H$ and $F$ is an $(D, 2\tau, \delta)$-pseudorandom hypergraph.
\end{lemma}

\begin{proof}
    Let $\psi: E(H) \to [0, 1]$ be the given $(d, \tau, \delta)$-pseudorandom fractional matching. We replace each edge $e\in E(H)$ by multiple edges of multiplicity $\lfloor \frac{D}{d}\psi(e) \rfloor$ and denote the resulting multi-hypergraph by $F$. 
    Then for every $v\in V(H)$, we have $d_F(v) = \sum_{v\in e\in E(H)} \lfloor \frac{D}{d}\psi(e) \rfloor = \frac{D}{d} d^{\psi}(v) \pm \Delta (H) = (1 \pm 2\tau)D$.
    Similarly, for every distinct two vertices $u, v\in V(H)$, we have $d_2(\{u, v\}) = \sum_{\{u, v\}\subseteq e \in E(H)} \lfloor \frac{D}{d}\psi(e) \rfloor \leq \delta D$. Thus, $F$ is a $(D, 2\tau, \delta)$-pseudorandom hypergraph. Since we replace edges in $E(H)$ with multiple edges, obviously $F_{\mathrm{simp}}$ is a sub-hypergraph of $H$.
\end{proof}

To prove \Cref{lem:cover-down}, we need the following two lemmas,

\begin{lemma}[{\cite[Theorem 7.1]{kang2020new}}]\label{lem:pseudorandom-alon-yuster-type}
    Let $k > 3$ be an integer, $D, \tau, \delta, \mu, \ve, \gamma, K > 0$ and $\ve \in (0, 1 - \frac{1}{k-1})$ be real numbers. Then there exists $n_0 = n_0(k, K, \gamma, \mu, \ve)$ such that the following holds for every $n \geq n_0$.

    Let $H$ be a $k$-uniform $n$-vertex $(D, \tau, \delta)$-pseudorandom multi-hypergraph. Assume $D \geq exp(\log^{\mu} n)$, $\delta \leq D^{-\gamma}$, and $\tau \leq K\delta^{1 - \ve}$. Let $\calF \subseteq 2^{V(H)}$ be a collection of vertex subsets of $V(H)$ such that $|F| \geq \delta^{-\frac{1}{2}}\log n$ for all $F\in \calF$ and $|\calF| \leq exp(\log^{\frac{4}{3}} n)$. Then $\calH$ has a $(\calF, \delta^{\frac{1}{k-1}})$-pseudorandom matching.
\end{lemma}

We note that \Cref{lem:pseudorandom-alon-yuster-type} is a simplified version of Theorem 7.1 in~\cite{kang2020new}, but this simpler version is sufficient for our purpose.
The following is the rainbow analogue of Lemma 3.10 in~\cite{barber2020minimalist}.

\begin{lemma}\label{lem:absorption-matching}
    Let $\mu$ be a positive real number and let $N$ be the positive integer. Then there exists $n_0$ which depends only on $\mu$ such that the following holds for all $n \geq n_0$. 

    Let $G_1, \dots, G_N$ be graphs on common vertex set $V$ of size $n$ and assume $A_1, \dots, A_N$ be the subset of $V$ that satisfies the followings:
    \begin{enumerate}
        \item[$\bullet$] $|A_i|$ is even and $\delta(G_i[A_i]) \geq (1/2 + 4\mu^{1/6})|A_i|$ for all $i \in [N]$,
        \item[$\bullet$] $|A_i| \geq \mu^{4/3}n$ for all $i\in [N]$,
        \item[$\bullet$] $|A_i \cap A_j| \leq \mu^2 n$ for all $1 \leq i < j \leq N$,
        \item[$\bullet$] every $v\in V$ is contained in at most $\mu n$ of the sets $A_i$.  
    \end{enumerate}
    Then for every $i\in [N]$, the graph $G_i[A_i]$ contains a perfect matching $M_i$ such that $E(M_i) \cap E(M_j) = \emptyset$ for every $1 \leq i < j \leq N$.
\end{lemma}

Since the proof of \Cref{lem:absorption-matching} is almost identical with the proof of Lemma 3.10 in~\cite{barber2020minimalist}, we omit the proof.
We are now ready to prove \Cref{lem:cover-down}.

\begin{proof}[Proof of \Cref{lem:cover-down}]
    Let $\delta \defeq \max \{\frac{3}{4}, \delta_f(3\ve)\}$ and $n' = |V\setminus U|$. Let $V = \{v_1, \dots, v_n\}$ and $V\setminus U = \{v_1, \dots, v_{n'}\}$.
    Choose an arbitrary color set $Y\subseteq X_9$ with $\ve^4 n^2 \leq |Y| \leq 2\ve^4 n^2$. Then we consider an auxiliary hypergraph $L$ on the vertex set $V$ with edges $f\in \binom{V}{3}$ such that $|\{i\in Y: f\in E(H_i)\}| \geq \ve^{5} n^2$. We observe the following.

    \begin{observation}\label{obs:L-codegree}
        $L^{(2)} = G$ and for every $e\in E(G)$, we have $d_L(e; U) \geq (\delta + 8\ve)|U|$. 
    \end{observation}

    \begin{claimproof}
        For all $e\in E(G)$, we have $$(\delta + 9\ve)|U| |Y| \leq \sum_{i\in Y} d_{H_i}(e; U) \leq d_L(e; U)|Y| + \ve^5 n^2|U|  \leq d_L(e; U)|Y| + \ve |U| |Y|.$$ This implies $d_L(e; U) \geq (\delta + 8\ve)|U|$ and $L^{(2)} = G$. This proves the observation.
    \end{claimproof}

    For a $3$-graph $H$ and two distinct vertices $u, v$ of $H$, we write $N_H(u v)$ as $N_H(\{u, v\})$.

    \begin{claim}\label{clm:local-cleaner}
        Let $\mu = \ve^{10}$.
        There exists subsets $A_1, \dots, A_{n'}$ in $U\setminus U'$ that satisfy the followings for all $i, j\in [n']$ with $i\neq j$. 
        \begin{enumerate}
            \item[$(a)$] $A_i \subseteq N_G(v_i)$,
            \item[$(b)$] $|A_i| \geq \frac{\mu |U|}{2}$,
            \item[$(c)$] $|A_i \cap A_j| \leq 2\mu^2 |U|$,
            \item[$(d)$] for all $u\in U$, we have $|N_L(u v_i)\cap A_i| \geq \frac{2|A_i|}{3}$,
            \item[$(e)$] for all $1 \leq i\neq j \leq n'$, we have $|N_L(v_i v_j) \cap A_i \cap A_j| \geq \frac{\mu^2 |U|}{8}$,
            \item[$(f)$] each $u\in U$ is contained in $A_i$ for at most $2\mu n$ numbers $i\in [n']$, and
            \item[$(g)$] the number $|E(G\setminus G[U])| - \sum_{i\in n'} |A_i|$ is divisible by $3$. 
        \end{enumerate}
    \end{claim}

    \begin{claimproof}
        Let $X_i$ be a random subset of $N_G(v_i)\cap (U\setminus U')$ obtained by choosing each vertex in $N_G(v_i)\cap (U\setminus U')$ with probability $\mu$ independently at random for each $i\in [n']$. We note that $(1-\ve)|U| \leq |U\setminus U'|$ and $(1 - 3\ve)|U| \leq |N_G(v_i)\cap (U\setminus U')| \leq |U|$. Then by Chernoff bound~(\Cref{lem:chernoff}) and the union bound, with probability at least $\frac{9}{10}$, we have the following for all $i\in [n']$.
        \begin{equation}\label{eq:eq1}
            |A_i| \geq \frac{2\mu|U|}{3}.
        \end{equation}
        
        Since $|N_G(v_i) \cap N_G(v_j) \cap (U\setminus U')| \geq (1 - 5\ve)|U|$ for every $1\leq i \neq j \leq n'$, by \Cref{lem:chernoff} and the union bound, with probability at least $\frac{9}{10}$, the following inequality holds for all $1\leq i\neq j\leq n'$.
        \begin{equation}\label{eq:eq2}
            |A_i \cap A_j| \leq 2\mu^2 n.   
        \end{equation}

        By \Cref{obs:L-codegree}, for all $i\in [n']$ and $u\in U$, we have $|N_L(\{u, v_i\})\cap N_G(v_i) \cap (U\setminus U')| \geq (\delta + 5\ve)|U|$. As $\delta \geq \frac{3}{4}$, again, by \Cref{lem:chernoff} and the union bound, with probability at least $\frac{9}{10}$, we have the following for all $u\in U$ and $i\in [n']$.
        \begin{equation}\label{eq:eq3}
            |N_L(\{u, v_i\})\cap A_i| \geq \frac{2|A_i|}{3}.
        \end{equation}

        Since for every $1\leq i \neq j \leq n'$, we have $|N_L(\{v_i, v_j\})\cap N_G(v_i) \cap N_G(v_j) \cap  (U\setminus U')| \geq (\delta + 3\ve)|U|$, we also have the following inequality with probability at least $\frac{9}{10}$ for every $1 \leq i \neq j \leq n'$.
        \begin{equation}\label{eq:eq4}
            |N_L(\{v_i, v_j\})\cap A_i\cap A_j| \geq \frac{\mu^2 |U|}{6}.
        \end{equation}

        Finally, by \Cref{lem:chernoff} and the union bound, with probability at least $\frac{9}{10}$, for every $u\in U$, the following holds.
        \begin{equation}\label{eq:eq5}
            |\{i\in [n']: u\in A_i\}| \leq 2\mu n.
        \end{equation}

        Then by the union bound, with positive probability, there exist the sets $A_1, \dots, A_{n'}$ that satisfy all \eqref{eq:eq1}---\eqref{eq:eq5}. To ensure they satisfy $(g)$, we remove at most two elements from $A_1$. Then the sets $A_1, \dots, A_{n'}$ satisfies all the conditions $(a)$ to $(g)$. This proves the claim. 
    \end{claimproof}

    Let fix the sets $A_1, \dots, A_{n'}$ that satisfy \Cref{clm:local-cleaner}. Let $R$ be a subgraph of $G$ such that $E(R) = \{v_iu\in E(G):i\in [n'], u\in A_i\}$.

    Let $G' =G \setminus (G[U]\cup R)$. For each $i\in [N]$, let $H'_i$ be a sub-hypergraph of $H_i$ such that $E(H_i') = \{\{u, v, w\}\subset E(H_i): uv, vw, wu\in E(G')\}$. Then $H_i'^{(2)} = G'$, so $\delta(H_i'^{(2)}) \geq (1 - 3\ve)n$. By $(g)$ in \Cref{clm:local-cleaner}, the number $|E(G')|$ is divisible by $3$.
    By \Cref{lem:pseudorandom-STS}, for every $i\in [N]$, there is $\left(1, 0, \frac{\log^2 n}{n} \right)$-pseudorandom $STS$ in $H'_i$. Let $Z$ be a subset of $[N]\setminus Y$ such that $[N]\setminus Z \subseteq X_9$ and $|Z| = \frac{|E(G')|}{3}$.
    We now consider an auxiliary $4$-uniform hypergraph $\widetilde{H}$ on the vertex set $E(G')\cup Z$ with the edge set $E(\widetilde{H}) = \{\{i\} \cup e: i\in Z, e\in E((H_i')_{aux})\}$.

    \begin{claim}\label{clm:color-hypergraph}
        $\widetilde{H}$ has a $(1, 0, \frac{\log^2 n}{n})$-pseudorandom fractional matching.
    \end{claim}

    \begin{claimproof}
        By \Cref{obs:correspondence}, there exists a $(1, 0 ,\frac{\log^2 n}{n})$-pseudorandom fractional matching $\psi_i$ in $(H_i')_{aux}$ for each $i \in Z$. We extend $\psi_i$ to $\widetilde{\psi}_i: E(\widetilde{H})\to [0, 1]$ in such a way that $\widetilde{\psi}_i(e) = \psi(e\setminus \{i\})$ whenever $i\in e$ and $\widetilde{\psi}_i(e) = 0$ otherwise. Let $\widetilde{\psi} \defeq \frac{1}{|Z|}\sum_{i\in Z}\widetilde{\psi}_i$. We now claim that $\widetilde{\psi}$ is a $(1, 0, \frac{\log^2 n}{n})$-pseudorandom fractional matching for $\widetilde{H}$.
        For every $e\in E(G') \subseteq V(\widetilde{H})$, we have $d^{\widetilde{\psi}}(e) = \frac{1}{|Z|}\sum_{e\subset f\in E(\widehat{H})} \widetilde{\psi}(f) = \frac{1}{|Z|}\sum_{i\in Z} \sum_{e\subset f'\in E((H_i')_{aux})} \psi_i(f') = \frac{1}{|Z|}|Z| = 1$. Meanwhile, for every $i\in Z$, the inequality $d^{\widetilde{\psi}}(i) = \frac{1}{|Z|} \sum_{e\subset f'\in E((H_i')_{aux})} \psi_i (f') = \frac{1}{|Z|}\frac{|E(G')|}{3} = 1$. Thus, $\widetilde{\psi}$ is a perfect fractional matching. 
        Let $e\neq e' \in E(G)$ and $i\neq j \in Z$. Then we have $d_2^{\widetilde{\psi}_i}(\{e, e'\}) = d_2^{{\psi}_i}(\{e, e'\}) \leq \frac{\log^2 n}{n}$. This implies $d_2^{\widetilde{\psi}}(\{e, e'\}) \leq  \frac{1}{|Z|}|Z|\frac{\log^2 n}{n} \leq \frac{\log^2 n}{n}$. We note that there is no hyperedge containing both $i$ and $j$. Finally, $d_2^{\widetilde{\psi}}(\{e, i\}) = \frac{1}{|Z|} d^{\widetilde{\psi}_i}(e) = \frac{1}{|Z|} \leq \frac{\log^2 n}{n}$. Thus, $\widetilde{\psi}$ is $(1, 0, \frac{\log^2 n}{n})$-pseudorandom fractional matching. This proves the claim.
    \end{claimproof}

    We note that $\widetilde{H}$ is a simple hypergraph,  thus, $\Delta(\widetilde{H}) \leq |Z|\binom{n}{2} \leq \frac{n^4}{12}$. By \Cref{lem:matching-to-hypergraph}, there is a multi-$4$-graph $\widehat{H}$ on the vertex set $E(G')\cup Z$ such that $\widehat{H}$ is a $(n^5, \frac{1}{n}, \frac{\log^2 n}{n})$-pseudorandom hypergraph and $\widehat{H}_{\mathrm{simp}} \subseteq \widetilde{H}$. 
    We claim that $\widehat{H}$ contains a matching $M_a$ such that $X_1 \subseteq V(M_a)$ and $|E(M_a)| = |X_1|$. Assume that $M$ is a matching in $\widehat{H}$ such that $|M \cap X_1| = i$ and $|E(M)| = i$ with maximum possible $i$. If $i < |X_1|$, choose a vertex $x\in X_1 \setminus M$ and consider the hypergraph $\widehat{H} - V(M)$. Since the maxmimum codegree of $\widehat{H}$ is bounded above by $n^4 \log^2 n$ and $i\leq \ve n^{1/4}$, the mimumum $1$-degree of $\widehat{H} - V(M)$ is at least $n^5 - n^4 - n^4\log^2 n > 0$. Thus, we can choose an edge $e_{i}$ in $\widehat{H} - V(M)$ that contains $x$. Then $M' = M \cup \{e\}$ contradicts the maximality of $i$.  Thus, we have $i = |X_i|$ and we obtain a matching $M_a = M$ such that $X_1 \subseteq V(M_a)$ and $|E(M_a)| = |X_1| \leq \ve n^{1/4}$.   

    Let $\widehat{H'}$ be the hypergraph $\widehat{H} - V(M_a)$. Since $|E(M_a)| \leq \ve n^{1/4}$ and the maximum codegree of $\widehat{H}$ is at most $n^4 \log^2 n$, the hypergraph $\widehat{H'}$ is a $(n^5, \frac{1}{\sqrt{n}}, \frac{\log^2 n}{n})$-pseudorandom hypergraph.
    Let $E_{v_i} \defeq N_{G'}(v_i) \setminus V(M_a)$.
    We note that for all $i\in [n]$, we have $n^{2/3} \leq |E_{v_i}| \leq n$. 
    By distributing vertices in $Z\cap X_9$ to $X_2, \dots, X_8$ in an appropriate way, we may assume $|X'_i| = \max\{\lceil 10n^{1/2}\rceil, |X_i| \}$ such that $X_i \subseteq X'_i$ for all $2 \leq i \leq 8$. 
    Let $\calF = \{X'_2, X'_3, \dots, X'_8, E_{v_1}, \dots, E_{v_{n}}\}$. We now apply \Cref{lem:pseudorandom-alon-yuster-type} on $\widehat{H'}$ and $\calF$ with parameters $4, n^5, \frac{2\log^2 n}{n}, \frac{\log^2 n}{n}, 1, \frac{1}{2}, \frac{1}{10}, 1$ playing the roles of $k, D, \tau, \delta, \mu, \ve, \gamma, K$, respectively. Then there exists a matching $M_b$ in $\widehat{H'}$ such that $|X'_i \setminus V(M_b)| \leq \frac{\ve^3}{10}|X'_i|n^{-1/4}$ for each $2\leq i \leq 8$. Thus, we have $|X_i \setminus V(M_b)| \leq \ve (\ve n)^{(i-1)/4}$. We also have $|E_{v_i}\setminus V(M_b)| \leq (\ve n)^{3/4}$. Let $M = M_a\cup M_b$. We observe that for every $e\in M$, if we let $e' = e \setminus Z$ and $i = e\cap Z$, then $e'$ is a triangle of $G$ and it comes from $H_i$. 
    
    Let $T_1$ be a $3$-graph on $V$ and $W' \subseteq Z$ such that $$E(T_1) = \{\{x, y, z\}\in V: \text{ there is } i\in Z \text{ such that } \{xy, yz, zx, i\} \in E(M)\},$$ $$W' \defeq \{i\in Z: \text{ there is } x, y, z\in V \text{ such that } \{xy, yz, zx, i\} \in E(M)\}.$$ Then each hyperedges in $T_1$ comes from $H_i$ for different $i\in Z$ and the maximum codegree of $T_1$ is at most $1$. Moreover, $X_1 \subseteq W'$, $|X_i\setminus W'| \leq \ve (\ve n)^{(i-1)/4}$ for each $2 \leq i \leq 9$. Let $G'' \defeq G'\setminus T_1^{(2)}$. Then we have $\Delta(G'') \leq (\ve n)^{3/4} = |U|^{3/4}$. For each $e = v_iv_j \in E(G''[V\setminus U])$, we will choose a vertex $u_{e}\in A_i\cap A_j$ such that $\{v_i, v_j, u_{e}\}\in E(L)$ and $u_e \neq u_{e'}$ whenever $e\cap e' \neq \emptyset$ for all $e'\in E(G''[V\setminus U])$. Since $2\Delta(G'') \leq 2|U|^{3/4}$ and by $(e)$ in \Cref{clm:local-cleaner}, we have $|N_L(\{v_i, v_j\}) \cap A_i \cap A_j| > 2|U|^{3/ 4}$ for all $1 \leq i \neq j \leq n'$, we can choose such $u_e$ greedily for all $e\in E(G''[V\setminus U])$.
    
    We denote as $T_2$ the hypergraph on the vertex set $V$ and the edge set $$E(T_2) \defeq \{e\cup \{u_e\}: e\in E(G''[V\setminus U])\}.$$ Let $R' \defeq (R \cup G'')\setminus T_2^{(2)}$. Observe that for every $i\in [n']$, we have $|N_{R'}(v_i)| \leq |A_i| + |U|^{3/4}$. Let $A'_i \defeq N_{R'}(v_i)$ for each $i\in [n']$. We also note that since $G$ was a $K_3$-divisible graph and $(T_1\cup T_2)^{(2)}$ is a $K_3$-divisible graph, the degree of $v_i$ of $R'$ is even for every $i\in [n']$. This implies that $|A'_i|$ is even for each $i\leq [n']$. For each $i\in [n']$, we define a graph $F_i$ on $U$ such that $$E(F_i) = \{e\in \binom{U}{2}\setminus \binom{U'}{2}: e\cup \{v_i\} \in E(L)\}$$ for each $i\in [n']$. Since $|A_i| - |U|^{3/4} \leq |A'_i| \leq |A_i| + |U|^{3/ 4}$, by $(a)$ and $(d)$ in \Cref{clm:local-cleaner}, we have $|A'_i| \geq \frac{\ve^{10}|U|}{3}$ and $\delta(F_i[A'_i]) \geq \frac{3|A'_i|}{5}$ for each $i\in [n']$. We also observe that by $(c)$ in \Cref{clm:local-cleaner}, the inequality $|A'_i \cap A'_j| \leq 2\ve^{20}|U| + 2|U|^{3/4} \leq 3\ve^{20}|U|$ holds for all $1 \leq i \neq j \leq n'$. Lastly, by $(g)$ in \Cref{clm:local-cleaner} and the fact that $\Delta(G'') \leq |U^{3/4}|$, every $u\in U$ contained in at most $2\ve^{10}|U| + |U|^{3/4} \leq 3\ve^{10}|U|$ of the $A'_i$'s. By applying \Cref{lem:absorption-matching} with parameter $2\ve^{9}$ playing a role of $\mu$, we obtain pairwise edge-disjoint matchings $J_1, \cdots J_{n'}$, where $J_i$ is a perfect matching in $F_i[A'_i]$ for each $i\in [n']$.
    
    Let $T_3$ be a $3$-graph on the vertex set $V$ such that $$E(T_3) \defeq \{e_i\cup \{v_i\}: i\in [n'],\text{ } e_i\in E(J_i)\}.$$ Let $T = T_1 \cup T_2 \cup T_3$. We note that $T_2\cup T_3 \subseteq L$ and $|E(T_2 \cup T_3)| \leq \ve^5 n^2$, by the definition of the hypergraph $L$, for each $e\in T_2\cup T_3$, we choose $Y_e\in Y$ greedily from $Y$ such that $Y_e \neq Y_{e'}$ for each $e, e'\in E(T_2\cup T_3)$ and $e\in E(H_{Y_e})$ whenever $e \neq e'$. Let $W'' \defeq \{Y_e\in Y: e\in E(T_2\cup T_3)\}$. We define $W \defeq W' \cup W''$. 
    
    We now claim that $T$ and $W$ are a desired hypergraph and a desired subset of $[N]$.
    By the construction of $T$, the maximum codegree of $T$ is $1$.
    Obviously, $(D5)$ of \Cref{lem:cover-down} holds for $T$ and $W$ by the construction above. Since for each $i\in [n']$, the graph $F_i$ does not containing any element of $\binom{U'}{2}$ as an edge, the graph $T^{(2)}$ is a subgraph of $G\setminus G[U']$. We observe that by the construction of $T$, the graph $T^{(2)}$ contains $G\setminus G[U]$. This implies that $(D2)$ of \Cref{lem:cover-down} holds.
    We note that we already proved that $X_1 \subset W'$ and $|X_i \setminus W'| \leq \ve (\ve n)^{(i-1)/4}$ for each $2\leq i \leq 8$. Since $|W| = |E(T)| \geq \frac{E(G\setminus G[U])}{3}$, we have $|X_9\setminus C| \leq |[N]\setminus C| \leq \frac{\ve^3 n^2}{2} \leq \ve (\lfloor \ve n \rfloor)^2$. Thus, $(D1)$ of \Cref{lem:cover-down} holds.
    By $(f)$ of \Cref{clm:local-cleaner}, we have $\Delta((T_2\cup T_3)^{(2)}[U]) \leq \Delta(G'') + 2\ve^{10}n \leq |U|^{3/4} + 3\ve^{9}|U| \leq \frac{\ve}{2}|U'|$.
    We observe that the followings $$\delta((G\setminus T^{(2)})[U]) = |U| - 1 - \Delta((T_2\cup T_3)^{(2)}[U]),$$ $$d_{(G\setminus T^{(2)})[U]}(u; U') \geq |U'| - 1 - \Delta((T_2\cup T_3)^{(2)}[U])$$ holds for all $u\in U$. Since $\Delta((T_2\cup T_3)^{(2)}[U]) \leq \frac{\ve}{2}|U'|$ holds, $(D3)$ and $(D4)$ of \Cref{lem:cover-down} hold. This completes the proof.
\end{proof}


\section{Proof of \Cref{thm:main}}\label{sec:proof}

We are now ready to prove \Cref{thm:main}.
Let $\calH = \{H_1, \dots, H_{\frac{n(n-1)}{6}}\}$ be a collection of $3$-graph on a common vertex set $V$.
To write the proof \Cref{thm:main} more clearly, the following definitions would be useful.
    
\begin{definition}
    Let $U\subseteq V$ and let $C\subseteq [\binom{n(n-1)}{6}]$.
    We say that a $3$-graph $F$ on the vertex set $U$ is \emph{$C$-rainbow} if there is a bijective function $\phi: E(F)\to C$ such that $e\in E(H_{\phi(e)})$ holds for every $e\in E(F)$.
\end{definition}

\begin{definition}
    For a subset $U\subseteq V$, a graph $G$ on the vertex set $U$, and $C\subset [\binom{n(n-1)}{6}]$ and a collection $\calH = \{H_i: i\in [\frac{n(n-1)}{6}]\}$, we denote by $\calH[U; G; C]$ a collection of hypergraphs $\{H'_i: i\in C\}$ such that $V(H'_i) = U$ and $E(H'_i) = \{\{x, y, z\}\in E(H_i[U]): xy, yz, zx \in E(G)\}$ for each $i\in C$.
\end{definition}

\begin{proof}[Proof of \Cref{thm:main}]
    For given $\ve > 0$, we fix the parameters $m'$ and $n$ as follows.
    $0 < \frac{1}{n} \ll \frac{1}{m'} \ll \ve < 1$, where $n \equiv 1 \text{ or } 3 \pmod{6}$.
    Let $\delta \defeq \max\{\frac{3}{4}, \delta_f(3\ve)\}$
    and $\calH = \{H_1, \dots, H_{\frac{n(n-1)}{6}}\}$ be a collection of $3$-graph on the common vertex set $V$ with size $n$ and for each $i\in [\frac{n(n-1)}{6}]$, we have $\delta_2(H_i) \geq (\delta + 100\ve)n$. Since $\delta_f(\gamma)$ goes to $\delta_f^*$ monotonically non-increasing as $\gamma$ goes to $0$, in order to prove \Cref{thm:main}, by readjusting the parameter $\ve$, it suffice to prove that there is a $\calH$-transversal $STS$ on $V$. Let $K$ be the complete graph on the vertex set $V$.

    \textbf{$\bullet$ Obtaining a transversal-vortex.} Let $m$ be an integer such that $\ve m' \leq m \leq m'$ and let $\ell \defeq \ell(n; m')$. We denote by $n_0, n_1, \dots, n_{\ell}$ the integers such that $n_0 = n$, $n_{i} = \lfloor \ve n_{i-1} \rfloor$ for each $i\in[\ell]$, and $n_{\ell} = m$. Let $(D_0, D_1, \dots, D_{\ell-1})$ be the partition of $[\frac{n(n-1)}{6}]$ with $|D_i| = \frac{1}{3}\left(\binom{n_i}{2} - \binom{n_{i+1}}{2}\right)$ for all $0\leq i \leq \ell-2$ and $|D_{\ell-1}| = \frac{1}{3}\binom{n_{\ell-1}}{2}$. For the simplicity of the notations, we set $D_i = \emptyset$ whenever $i < 0$. We let $C_i \defeq \bigcup_{i-10 \leq j \leq \ell-1} D_j$ for each $0\leq i \leq \ell-1$. Then $\calC = \{C_0 \supseteq C_1 \supseteq \cdots \supseteq C_{\ell -  1}\}$ is a $(\binom{n(n-1)}{2}, \ve)$-color-vortex. By \Cref{lem:vortex-transversalSTS}, $\calH$ has a $(\delta + 99\ve, \ve, m)$-transversal-vortex $U_0 \supseteq U_1 \supseteq \cdots \supseteq U_{\ell}$ with respect to $\calC$.

    \textbf{$\bullet$ Placing a color-absorber.} Let $\calT$ be the collection of all $K_3$-divisible graphs on $U_{\ell}$. Then for each $T\in \calT$, \Cref{lem:sparse-K3-absorber} implies that there exists a $K_3$-absorber $A_T$ for $T$ whose $T$-rooted edge-degeneracy is at most $4$. Let $$t \defeq \max\{|V(A_T)| - |V(T)|: T\in \calT\},$$ $$a_T \defeq \frac{1}{3}|E(A^{(2)}_T \setminus T)|,$$ $$b_T \defeq \frac{1}{3}|E(T)|, \text{ and}$$ $$\calP \defeq \{(T, Z): T\in \calT, Z\subseteq \binom{C_{\ell - 1}}{b_T}\}.$$
    Observe that the $|\calP|$ is at most $2^{\binom{m}{2}}2^{|C_{\ell - 1}|} \leq \frac{1}{(t+m)^3}\ve^{100}n$. Now, for each $(T, Z)\in \calP$, we assign a set $W(T, Z) \subseteq D_0$ of numbers with $|W(T, Z)| = a_T$ in such a way that the sets $W(T, Z)$ are pairwise disjoint. Indeed, this is possible because $\sum_{(T, Z)\in \calP}a_T \leq (t + m)^3 |\calP| \leq \ve^{100}n \leq |D_0|$. With these choices, the following properties hold for all $(T, Z) \neq (T', Z') \in \calP$. 

    \begin{enumerate}
        \item[$(i)$] $A_{(T, Z)}$ is a $K_3$-absorber for $T$ and has $V(T)$-rooted edge-degeneracy at most $4$,
        
        \item[$(ii)$] $V(A_{(T, Z)}) \setminus V(T) \subseteq U_0\setminus U_1$,

        \item[$(iii)$] $(A_{(T, Z)})^{(2)}\setminus T^{(2)}$ and $(A_{(T, Z)})^{(2)}$ has both $W(T, Z)$-rainbow and $(W(T, Z)\cup Z)$-rainbow $K_3$-decompositions in $A_{(T, Z)}$, respectively,

        \item[$(iv)$] $(V(A_{(T, Z)}) \setminus V(T)) \cap (V(A_{(T', Z')}) \setminus V(T')) = \emptyset$.
        
    \end{enumerate}

    Let $A_0\defeq \bigcup_{(T, Z)\in \calP} (A_{(T, Z)})^{(2)}\setminus T^{(2)}$ and $W_0 \defeq \bigcup_{(T, Z)\in \calP} W(T, Z)$. We note that $|V(A_0)| \leq \ve^{100}n$. Then by $(i)$, the graph $A_0[U_1]$ has no edge. Moreover, we have the following by $(i)-(iv)$.

    \begin{equation}\label{prop:color-absorbing}
        \text{For all $T\in \calT$ and for all $Z\subseteq \binom{C_{\ell-1}}{b_T}$, the graph $A_0\cup T$ has a $(W_0\cup Z)$-rainbow $K_3$-decomposition.}\tag{$\clubsuit$}
    \end{equation}
    
    We note that \eqref{prop:color-absorbing} directly followed from $(i)-(iv)$. Let $G_0 \defeq K\setminus A_0$ and let $D'_0 \defeq D_0\setminus W_0$.  Obviously, $A_0$ has a $W_0$-rainbow $K_3$-decomposition by \eqref{prop:color-absorbing}. We note that $G_0$ is a $K_3$-divisible graph and $G_0[U_1]$ is a complete graph.

    \textbf{$\bullet$ Iteratively applying \Cref{lem:cover-down}.} 
    Let $S_0$ be an empty set and $S_1\subseteq D_1$ be an arbitrary subset such that $$\frac{|E(G\setminus G[U_1]])|}{3} + \ve^4 n^2 \leq |D'_0| + |S_1| \leq \frac{|E(G\setminus G[U_1]])|}{3} + \frac{\ve^3 n^2}{2}.$$ 
    Since $|V(A_0)| \leq \ve^{100}n$,  we have $\delta(G_0) \geq (1 - \ve)n$ and $d_{G_0}(v; U_1) \geq (1 - 2\ve)|U_1|$. We now initialize $X^{(0)}_j = \emptyset$ for all $1\leq j \leq 8$ and let $X^{(0)}_9 = D'_0 \cup S_1$. Since $\delta(G_0) \geq (1 - \ve)n$, for every hypergraph $H\in \calH[U_0; G_0; D'_0\cup S_1]$, we observe that $H^{(2)} = G$ and $\delta_{\mathrm{ess}}(H) \geq (\delta + 90\ve)n_0$ and $d_H(e; U_1) \geq (\delta + 80 \ve)n_1$ for all $e\in E(G_0)$. Then we can apply \Cref{lem:cover-down} to $\calH[U_0; G_0; D'_0\cup S_1]$, $U_0$, $U_1$, $G_0$, $X^{(0)}_1, \dots, X^{(0)}_9$ which playing the roles of $\calH$, $V$, $U$, $G$, $X_1\cdots, X_9$, respectively. We claim that we can iteratively apply \Cref{lem:cover-down} to cover the edges of $G_0$ until all the remaining uncovered edges are contained in $U_{\ell}$.

    Assume for some $0\leq i \leq \ell - 1$, we have $D'_i$, $S_i$, $A_i$, $W_i$, $X^{(i)}_1, \cdots, X^{(i)}_8$, where $S_i\subseteq D_i$, $D'_i = D_i\setminus S_i$, that satisfy the followings.
    \begin{enumerate}
        \item[$(a)$] $\bigcup_{j \leq i - 9} D_j \subseteq W_i \subseteq \bigcup_{j \leq i} D_j$,
        \item[$(b)$] $A_0$ and $A_i$ are edge-disjoint and $K\setminus K[U_{i}] \subseteq A_i$,
        \item[$(c)$] there is a $W_i$-rainbow $K_3$-decomposition of $A_i$,
        \item[$(d)$] let $G_i \defeq (K \setminus A_i)[U_{i}]$ then $\delta(G_i) \geq (1 - \ve)n_i$, 
        \item[$(e)$] $G_i[U_{i+1}]$ is a complete graph,
        \item[$(f)$] $S_i \subseteq W_i \cup X^{(i)}_8$ and $D'_i$ is disjont from $W_i\cup \bigcup_{j\in [8]} X^{(i)}_j$,
        \item[$(g)$] $|X^{(i)}_j| \leq \ve n_i^{j/4}$ for all $j\in [8]$ and for each $k\in [8]$, the set $D_{i - k} \setminus W_i$ is contained in $X^{(i)}_{9-k}\cup X^{(i)}_{8-k}$ (we may consider $X^{(i)}_0 = \emptyset$).
    \end{enumerate}
    Since $U_0 \supseteq U_1 \supseteq \cdots \supseteq U_{\ell}$ is a $(\delta + 99\ve, \ve, m)$-transversal-vortex with respect to $\calC$, together with $(a)$ and $(d)$, for all hypergraph $H\in \calH[U_{i};G_i;C_i]$, we have $H^{(2)} = H_i$ and its essential minimum codegree is at least $(\delta + 90\ve)n_i$. Moreover, we have $d_{H}(e; U_{i+1}) \geq (\delta + 80\ve)n_{i+1}$ for all $e\in E(G_i)$.
    By $(c)$, the graph $G_i$ is $K_3$-divisible. By $(a), (c), (f)$ and by our choice of $D_i$, the following inequality $$\frac{1}{3}|E(G_i \setminus G_i[U_{i+1}])| - 1 \leq |D'_i| + \sum_{j\in [8]}|X^{(i)}_j| \leq \frac{1}{3}|E(G_i \setminus G_i[U_{i+1}])| + 1$$ holds. 
    Thus, we can choose a set $S_{i+1} \subseteq D_{i+1}$ such that $2\ve^4 n_i^2 \leq |S_{i+1}| \leq \frac{\ve^3 n_i^2}{3}$. Let $X^{(i)}_9 \defeq D'_i \cup S_{i+1}$. Then by applying \Cref{lem:cover-down} to $\calH[U_i;G_i;\bigcup_{j\in [9]} X^{(i)}_j]$, we obtain a set $W'_{i+1} \subseteq \bigcup_{j\in [9]} X^{(i)}_j$ and a subgraph $A'_{i+1}$, which is contained in $G_i\setminus G_i[U_{i+1}]$ and contains $G_i \setminus G_i[U_{i+1}]$ as a subgraph, satisfying the following. 

    \begin{enumerate}
        \item[$(1)$]  $X^{(i)}_1 \subseteq W'_{i+1}$ and $|X^{(i)}_j \setminus W'_{i+1}| \leq \ve n_{i+1}^{(j-1)/4}$ for each $2 \leq j \leq 9$,
        \item[$(2)$] $A'_{i+1}$ has a $W'_{i+1}$-rainbow $K_3$-decomposition,
        \item[$(3)$] we have $\delta(G_i\setminus A'_{i+1}[U_{i+1}]) \geq (1 - \ve)n_{i+1}$ and $d_{(G_i\setminus A'_{i+1})[U]}(u; U_{i+1}) \geq (1 - 2\ve)n$ for all $u\in U_i$.
    \end{enumerate}
    
    We now define $$W_{i+1}\defeq W_i\cup W'_{i+1},$$ $$A_{i+1} \defeq A_i\cup A'_{i+1},$$ $$D'_{i+1} \defeq D_{i+1} \setminus S_{i+1}.$$ We also update the set $X^{(i+1)}_j = X^{(i)}_{j+1} \setminus W'_{i+1}$ for each $j\in [8]$. Then $D'_{i+1}, S_{i+1}, A_{i+1}, W_{i+1}, X^{(i+1)}_1, \dots, X^{(i+1)}_8$ satisfies all $(a)$ to $(g)$, where $i+1$ plays a role of $i$. 
    
    Iterating this procedure yields a graph $A_{\ell}$ and a set $W_{\ell}$ and a $W_{\ell}$-rainbow $K_3$-decomposition on $A_{\ell}$ such that $K\setminus K[U_{\ell}] \subseteq A_{\ell}$ and $[\frac{n(n-1)}{6}]\setminus W_{\ell} \subseteq C_{\ell - 1}$.

    \textbf{$\bullet$ Absorbing the remaining structures and colors.} We note that $A_{\ell} = A_0 \cup \left(\bigcup_{i\in [\ell]} A'_{i}\right)$ and the sets $A_0, A'_1, \dots, A'_{\ell}$ are pairwise edge-disjoint. We also remark that $(W_0, W'_1, \dots, W'_{\ell})$ is a partition of $W_{\ell}$ and for each $i\in [\ell]$, there is a $W'_i$-rainbow $K_3$-decomposition on $A'_{i}$. This implies that there is a $(W_{\ell} \setminus W_0)$-rainbow $K_3$-decomposition on $A_{\ell} \setminus A_0$.
    Let $T \defeq K\setminus A_{\ell}$ and let $Z \defeq [\frac{n(n-1)}{6}] \setminus W_{\ell}$. As $(T, Z)\in \calP$, \eqref{prop:color-absorbing} implies that there is a $(W_0\cup Z)$-rainbow $K_3$-decomposition on $A_0\cup T$. Together with the $(W_{\ell} \setminus W_0)$-rainbow $K_3$-decomposition on $A_{\ell} \setminus A_0$, we obtain a desired $\calH$-transversal $STS$.
\end{proof}


\section{Concluding remarks}\label{sec:concluding}

In this article, we obtained an upper bound of the minimum codegree threshold not only for the existence of $STS$ but also for the existence of transversal $STS$. We also prove that the lower bound for the threshold is at least $\frac{3}{4}n - 26$ and we believe that the correct threshold would be $\frac{3}{4}n + C$ for some constant $C$. Thus, we leave it as a conjecture.

\begin{conjecture}\label{conj:3/4}
    There is a constant $n_0$ and $C$ such that the following holds for all $n\geq n_0$ satisfying $n \equiv 1 \text{ or } 3 \pmod{6}$.
    Let $\calH = \{H_1, \dots, H_{\frac{n(n-1)}{6}}\}$ be the collection of $3$-graphs on the common $n$-vertex set $V$. Assume for each $i\in [\frac{n(n-1)}{6}]$, the minimum codegree of $H_i$ is at least $\frac{3}{4}n + C$. Then there is a $\calH$-transversal $STS$ on $V$.
\end{conjecture}

\Cref{thm:main} implies that, in order to prove an asymptotic version of \Cref{conj:3/4}, it suffices to show $\delta^*_f \leq \frac{3}{4}$. We further conjecture that the same minimum codegree condition implies the existence of the resolvable $STS$ and the transversal copy of resolvable $STS$. Here, the resolvable $STS$ is a $STS$ that all triples in the system can be perfectly decomposed into perfect matchings. The resolvable $STS$ exists on $n$ vertices if and only if $n \equiv 3 \pmod{6}$~\cite{ray1971solution}. 

\begin{conjecture}
    There is a constant $n_0$ and $C$ such that the following holds for all $n\geq n_0$ satisfying $n \equiv 3 \pmod{6}$.
    Let $H$ is a hypergraph on $n$ vertices with minimum codegree at least $\frac{3}{4}n + C$. Then $H$ contains a resolvable $STS$.
\end{conjecture}

\begin{conjecture}
    There is a constant $n_0$ and $C$ such that the following holds for all $n\geq n_0$ satisfying $n \equiv 3 \pmod{6}$.
    Let $\calH = \{H_1, \dots, H_{\frac{n(n-1)}{6}}\}$ be the collection of $3$-graph on the common $n$-vertex set $V$. Assume for each $i\in [\frac{n(n-1)}{6}]$, the minimum codegree of $H_i$ is at least $\frac{3}{4}n + C$. Then there is a $\calH$-transversal resolvable $STS$ on $V$.
\end{conjecture}

Very recently, Zheng~\cite{zheng2024codegree} improved the upper bound of $\delta_f^*$ to $0.8579$. Together with \Cref{thm:main}, the codegree threshold of the existence of $STS$ in sufficiently large $n$-vertex $3$-graph is at most $(0.8579 + o(1))n$.

\subsection*{Acknowledgement}
The author thanks his advisor Jaehoon Kim for his helpful comments and advice. He also thanks Michael Zheng for pointing out several minor errors and typos in the early version of this paper. 

\printbibliography

\end{document}